\documentclass[a4paper,10pt]{amsart}

\usepackage{amsmath}
\usepackage{amssymb}
\usepackage{amsthm}
\usepackage[english]{babel}
\usepackage{hyperref}

\date{}

\numberwithin{equation}{section}
\newtheorem{defi}{Definition}[section]
\newtheorem*{defi*}{Definition}
\newtheorem{rem}[defi]{Remark}

\newtheorem{prop}[defi]{Proposition}
\newtheorem{cor}[defi]{Corollary}
\newtheorem{lem}[defi]{Lemma}
\newtheorem{thm}[defi]{Theorem}

\newcommand{\ep}{\varepsilon}

\newcommand{\N}{\mathbb{N}}
\newcommand{\Z}{\mathbb{Z}}
\newcommand{\R}{\mathbb{R}}
\newcommand{\C}{\mathbb{C}}

\newcommand{\mf}{\mathcal{F}}

\renewcommand{\phi}{\varphi}
\renewcommand{\epsilon}{\varepsilon}

\newcommand{\Real}{\operatorname{Re}}
\newcommand{\im}{\operatorname{Im}}

\newcommand{\wa}{\alpha}
\newcommand{\betaw}{\beta}

\newcommand{\beq}{\begin{equation}}
\newcommand{\eeq}{\end{equation}}
\newcommand{\RE}{\mathbb{R}}
\newcommand{\CO}{\mathbb{C}}
\newcommand{\NN}{\mathbb{N}}
\renewcommand{\Im}{\operatorname{Im}}
\renewcommand{\Re}{\operatorname{Re}}
\newcommand{\f}{\frac}
\newcommand{\al}{\alpha}

\renewcommand{\ln}{\log}

\begin{document}

\title{The Riemann zeta in terms of the dilogarithm}

\author[S. Albeverio]{Sergio Albeverio}
\address{
Institut f\"{u}r Angewandte Mathematik,
Universit\"{a}t Bonn,
Endenicher Allee 60,
53115 Bonn, Germany;
SFB611, Bonn; HCM, Bonn; BIBOS, Bielefeld
and Bonn; IZKS, Bonn; CERFIM (Locarno);
Acc. Arch., USI (Mendrisio); Dip. Matematica Universit\`a di Trento
(Italy)}
\email{albeverio@uni-bonn.de}

\author[C. Cacciapuoti]{Claudio Cacciapuoti}
\address{
Hausdorff Center for Mathematics, Institut f\"{u}r Angewandte Mathematik,
Universit\"{a}t Bonn, Endenicher Allee 60,
53115 Bonn, Germany;
Graduate School of Science,
Mathematical Institute, Tohoku University, 
 6-3 Aoba, Aramaki, Aoba,
Sendai, Miyagi 980-8578, Japan
}
\email{cacciapuoti@hcm.uni-bonn.de}

\thanks{This paper is part of a larger research program. Previous collaborations and discussions 
with Zdzis{\l}aw Brze\'zniak, Christof
Cebulla, Benedetta Ferrario, Andrej Madrecki, Danilo Merlini, Lev Pustylnikov, Walter R.
Schneider and Boguslav Zegarli\'nski on topics related to the present
work helped us in understanding some of the underlying problems
and are gratefully acknowledged. S.A. is also 
grateful to Christoph Berns and
Masha Gordina for a careful reading of versions of a previous manuscript,
pointing out several necessary improvements, and making many very
useful comments for improving it. 
We are also very grateful to an anonymous referee for very careful reading of the paper, for many suggestions and much constructive criticism, which also greatly helped to improve the presentation of the results.\\
The warm hospitality and support of the Department of
Mathematics, University of Trento and of C.I.R.M., 
the Scuola Normale Superiore and the Centro de Giorgi in Pisa,
are also
gratefully acknowledged. One of the authors, C.C., is thankful to the
Hausdorff Research Institute for Mathematics and to the JSPS (Japan
Society for the
Promotion of Science) for financial support. 
We also thank Dominik Gunkel for his skillful help in type setting the manuscript.
}



\begin{abstract}
We give a representation of the classical Riemann $\zeta$-function in the half
plane $\Re s>0$ in terms of a Mellin transform involving the real part of the
dilogarithm function with an argument on the unit circle (associated Clausen
$Gl_2$-function). We also derive corresponding representations involving the
derivatives  of the $Gl_2$-function. A generalized symmetrized M\"untz-type formula is also derived. For  a special choice of test functions it connects to our integral representation of the $\zeta$-function, providing also a computation of a concrete Mellin transform. Certain formulae involving series of zeta functions and gamma functions are also derived.  
\end{abstract}

\maketitle

\section{Introduction}
\label{sec1}

The classical Riemann zeta function $\zeta$ is one of the
most intriguing and central objects of mathematics. Both the location and
distribution of its zeros have deep connections with the distribution of
prime numbers. The set  of zeros of $\zeta$ consists exclusively   of the
strictly negative even integers (``trivial zeros'') and, according to
the  as of 
yet still unproven Riemann's hypothesis \cite{ri}, of a  countable number of points on
the critical line $\Real s=\frac12$ in the complex $s$-plane (that there
are indeed infinitely many zeros on the critical line is a classical
result of Hardy, see, e.g., \cite{Edw74}). There are many connections of Riemann's zeta function with areas of mathematics and its applications, as well as with other conjectures, see
\cite{ri} for the original work by B. Riemann (and also, e.g., \cite{BCRW, Edw74} for comments resp. further historical comments) and \cite{ivi,
  KV92, pat, tit} for basic specific books on the Riemann's zeta function, as
well as the survey papers  \cite{bomb, con2}. For connections with other
problems in analytic number theory see, e.g., \cite{bru, jam}, for relations
with new developments in random matrix theory and other areas of mathematics
see, e.g., \cite{cm07a, con2, con, lap, RSdmj8196}, for numerical results and
other relations to methods inspired by physics see \cite{bms08}, and references therein. Some results establishing ``zero-free regions'' in 
$\left\{\Real s\neq\frac12\right\}\backslash\{-2\N\}$
are known, see, e.g., \cite{albce, Fo02, jam, KV92, tit, frank}, and references therein. These, in turn, are related with estimates on the remainder in the classical prime number theorem see, e.g., \cite{jam, KV92}.

The dilogarithm function $Li_{2}(z)$, $z\in\C$, was introduced by Euler
in $1768$ by its series representation $\sum_{n=1}^\infty\frac{z^n}{n^2}$ (for
$|z|\leqslant1$) and further studied, e.g. by Hill in 1828 see, e.g.,
\cite{lew, lew81}. It is a member of a family of polylogarithm functions which have  many connections with other areas of mathematics, like number theory, differential and algebraic geometry, but also with physics (quantum electrodynamics, scattering theory),  and engineerings (signal analysis, electrical network theory,  network problems, waveguide theory).
An important function associated with $Li_2$ is Clausen's $Gl_2$-function defined by 
\[
\Real Li_2(e^{i\theta})=Gl_2(\theta)=\sum_{n=1}^{\infty}\frac{\cos n\theta}{n^2}.
\] 
It is given in $[0,2\pi]$ by $\frac{\pi^2}{6}-\frac{\theta(2\pi-\theta)}{4}$.
In the present paper we relate $Gl_2(\theta)$, which we call for simplicity
$p(\theta)$, to the $\zeta$-function, via a Mellin transform. 
More precisely,
we prove a representation of $\zeta(s)$ for $\Re s>0$ in terms of $p(\theta)$
of the form
\[
\zeta(s)=\frac{ 2s(1+s)}{(2\pi)^{1-s}}
\left\{
\frac{\pi^2}{6}\frac1{1+s}-\frac{\pi}{2}\frac1s
-\frac14 \frac1{1-s}-D(-2-s)\right\}\,,
\]
where $D(\alpha)$ is defined for $\Re\alpha<-2$ as
$D(\alpha)\equiv\int_1^{+\infty}y^{\alpha}p(y)dy$ (see Theorem
\ref{t.3.1}) and is the Mellin transform at
$\alpha+1$ of the function $p_0(y)=\chi_{[1,\infty)}(y)p(y)$
($\chi_A$ denoting the characteristic function of the set $A$).  This provides a, to the best of our knowledge, new integral representation of
 $\zeta(s)$ for $\Re s>0$. Furthermore we
provide a series representation for the function $D(\alpha)$ (see Remark \ref{r.2.4}). We note that our
representation of the $\zeta$-function differs in several ways from other known
representations, like those given in \cite{bomb, Edw74, ri, RSdmj8196, srich}
(see also \cite[Sect. 3]{Ten}). As corollaries we obtain a new proof that $\zeta$ is a meromorphic function in $\Real s>0$, provide new proofs of certain
results on zero-free regions for the $\zeta$-function and bounds of it inside
the critical strip. 
We also derive other integral representations of $\zeta$ in terms of derivatives
of the function $p$ (Section 4).

Our integral representation of the $\zeta$-function also yields an explicit
formula for the Mellin transform of the Fourier transform of the test function
$\varphi(x)=(1-|x|)\chi_{[-1,1]}(x)$ (see Remark \ref{r.5.11}). We also get an explicit formula  for the Mellin transform of the function
$\frac{1-\cos(2\pi x)}{2\pi^2x^2}\chi_{[0,1]}(x)$  in terms of
$\Gamma$- functions. The summation of
certain series involving factorial factors (see Section \ref{sec5})
resp. the zeta functions at equally spaced arguments is also performed as an application of our integral representation (see Appendix \ref{sec6}).

The structure of the present paper is as follows. In Section \ref{sec2} after the introduction of 
$p(\theta)=Gl_2(\theta)$ we prove a lemma giving a representation of 
$
D_N(\alpha)\equiv\int_1^{2\pi N}y^{\alpha}p(y)dy
$
for $\Real \alpha<-1,\ N\in\N$ (incomplete Mellin transform of $p$). We then prove that 
$ \lim_{N\to\infty}D_N(\alpha)\equiv D(\alpha)$ exists and can be expressed by $\zeta(-2-\alpha)$ for $\Re \alpha<-2$. Then expressions for $D(\alpha)$ as series involving the gamma function are given. In Section \ref{sec3} the representation of $\zeta$ in terms of $D$ is used in particular to derive in a simple way a zero-free region for $\zeta$, close to the real line  (cf. Remark \ref{r.3.3}). In Section \ref{sec4} we derive two  other representations of $\zeta$ in terms of integrals involving, instead of $p$,
its distributional first derivative (Proposition \ref{prop-tildeE}) resp. a piecewise constant function (Proposition \ref{prop-F}).
In particular they lead to upper bounds on
$|\zeta(s)|$, for $\Real s>0$.
For a comparison with other bounds obtained essentially by trigonometric sums methods see Remark \ref{rem.4.4} below.

In Section \ref{sec5} we derive a generalized M\"untz formula for $\zeta$,
relative to general ``test functions'' which are such that they as well as
their Fourier  transforms are in $L^1(\R)$ (for more restrictive choices of  $f$ the formula was originally proven in \cite{mun},
see also \cite{bur}, and Remark \ref{rem22} below). We also prove,  exploiting
a Poisson summation formula, a symmetrized version of our generalized M\"untz
formula, somewhat related to the one discussed in \cite{albce}, and which
might be of interest in itself. For the special choice $f=\phi$ (with $\phi$ as above)these formulae
yield an explicit  computation of the Mellin transform of the Fourier transform of $f$ in terms of $\zeta$. 

In Appendix \ref{sec6} we provide another derivation for the basic function $D$ of Section \ref{sec2}. This derivation involves the computation of certain series $\tilde A(\alpha)$ (simply related to $D(\alpha)$) in terms of series involving the $\zeta$-functions taken at equally spaced arguments (cf. Lemma \ref{prop6.2}). The latter in turn are expressed in Corollary \ref{cor.6.5} in simple terms and a zeta function at a single point. These relations might have an interest in themselves (in any case we were not able to locate them in the extensive survey  on series involving the zeta function presented  in \cite{srich}).

\section{An integral involving the associated Clausen $Gl_2$-function}
\label{sec2}

Let $p$ be the $2\pi$-periodic real-valued function on $\R$ given for $\theta\in[0,2\pi]$ by 
\begin{align}\label{e.2.1}
p(\theta)=\frac{\pi^2}{6}-\frac{\theta(2\pi-\theta)}{4}.
\end{align}

The function $p(\theta)$
is denoted by $Gl_2(\theta)$ in \cite[p. 181]{lew} and is called the associated 
Clausen function (of order 2). It is also given by
\beq
\label{e.2.2}
p(\theta)=\sum_{n=1}^{+\infty}\frac{\cos (n\theta)}{n^2},\qquad\theta\in\R,
\eeq
see, e.g., \cite[p. 242]{lew}.
One has
\beq
\label{e.2.3}
p(\theta)=\pi^2 B_2\Big(\frac{\theta}{2\pi}\Big),\qquad \theta\in[0,2\pi],
\eeq
where $B_2(x)$ is the second Bernoulli polynomial,
defined as the coefficient of $t^2/2!$
in the expansion of $t e^{tx}/(e^{t}-1)$ in powers of $t$, i.e.
\begin{align}
B_2(x)=x^2-x+\frac16
\end{align}
(cf., e.g., \cite[p. 361]{erd}, and \cite[p.186]{lew}).

$p$ coincides with the real part of Euler's (1768) dilogarithm function
$Li_2(\cdot)$ with an argument on the unit circle, i.e.
\begin{align}
p(\theta)=\Real Li_2(e^{i\theta}), \qquad \theta\in[0,2\pi],
\end{align}
$Li_2(z)$ being defined for $|z|\leqslant1$ by
\begin{align}
Li_2(z)\equiv\sum_{n=1}^{+\infty}\frac{z^n}{n^2}
\end{align}
(see, e.g.,  \cite{lew}, \cite{lew81} and   \cite[p. 106]{tat}).

$p$ is bounded continuous on $\R$ (with $\max
p(\theta)=\max|p(\theta)|=\pi^2/6,\ 
\min p(\theta)=-\pi^2/12$). All its derivatives exist and are
continuous except for the points $2\pi k,\ k\in\Z$
(where they have to be defined, e.g., in the distributional sense).
The function $y\to y^{\alpha}p(y)$ is in $L^1([1,+\infty))$ (with respect to Lebesgue's measure
on $[1,+\infty)$), for any $\alpha\in\C$ with $\Real\alpha<-1$ 
(since $p$ is bounded).

One has the integral representation
\[
p(x)=\Real Li_2(e^{ix})
=-\frac{1}{2}\int_0^1\frac{1}{y}\log(1-2y\cos x+y^2)dy,
\]
see, e.g., \cite[p. 106]{lew}.

We shall study the function
\begin{align}\label{e.2.7}
D(\alpha)\equiv\int_1^{+\infty}y^{\alpha}p(y)dy,
\end{align} 
relating it, in Section \ref{sec3}, to the classical Riemann zeta function $\zeta$ at 
$(-2-\alpha),\ \Real\alpha<-2$
(i.e. $\zeta(s)$ with $\Real s>0$).

%%%%%
%LEMMA
%%%%%
\begin{lem}
\label{lem-DNal}
For any integer $N\in\NN$ and $\Re \al<-1$,  let $D_N(\alpha)$ be defined by
\beq
\label{DNalpha}
D_N(\alpha) \equiv \int_1^{2\pi N}y^{\alpha}p(y)dy,
\eeq
then
\beq
\label{res-DNal}
\begin{aligned}
D_N(\alpha)=&-\frac{\pi^2}{6}\frac1{\alpha+1}
+\frac{\pi}{2}\frac1{\alpha+2}
-\frac14\frac1{\alpha+3}-\frac{(2\pi)^{\al+3}}{2(\al+2)(\al+1)}\zeta_{N}(-2-\al)\\
&+\frac{\pi^2}{6}\frac{(2\pi N)^{\al+1}}{\al+1}
+\frac{\pi}{2}\frac{(2\pi N)^{\al+2}}{(\al+2)(\al+1)}\\
&+\frac{1}{2}\frac{(2\pi N)^{\al+3}}{(\al+3)(\al+2)(\al+1)},
\end{aligned}
\eeq
with
\beq
\label{zetaN}
\zeta_{N}(s) \equiv \sum_{n=1}^{N}\frac1{n^s}.
\eeq
\end{lem}

\begin{proof}
Splitting the integration domain $[1,2N\pi)$ in \eqref{DNalpha} into
\[
[1,2\pi)\cup[2\pi,2N\pi)=[1,2\pi)\cup\bigcup_{k=1}^{N-1}[2\pi k,2\pi (k+1))
\]
we get
\beq
\label{splitDNalpha}
\int_1^{2N\pi}y^{\wa}p(y)dy
=\int_1^{2\pi}y^{\wa}p(y)dy
+\sum_{k=1}^{N-1}\int_{2\pi k}^{2\pi(k+1)}y^{\wa}p(y)dy.
\eeq
Set
\beq
\label{Ialpha}
I^{\alpha}\equiv\int_1^{2\pi}y^{\wa}p(y)dy.
\eeq
Then, by the definition \eqref{e.2.1} of $p$ in $[0,2\pi]$:
\beq
\begin{aligned}
\label{Ialpha-2}
I^{\alpha}&=\int_1^{2\pi}y^{\wa}\Big(\frac{\pi^2}{6}-\frac{\pi}{2}y+\frac14y^2\Big)dy
=\left.\frac{\pi^2}{6}\frac{y^{\wa+1}}{\wa+1}\right|_1^{2\pi}
-\left.\frac{\pi}{2}\frac{y^{\wa+2}}{\wa+2}\right|_1^{2\pi}
+\left.\frac{1}{4}\frac{y^{\wa+3}}{\wa+3}\right|_1^{2\pi}\\
&=\frac{\pi^2}{6}\frac1{\wa+1}((2\pi)^{\wa+1}-1)
-\frac{\pi}{2}\frac1{\wa+2}((2\pi)^{\wa+2}-1)
+\frac{1}{4}\frac1{\wa+3}((2\pi)^{\wa+3}-1).
\end{aligned}
\eeq
Set
\beq
\label{IINalpha}
II^{\alpha}_N\equiv\sum_{k=1}^{N-1}\int_{2\pi k}^{2\pi(k+1)}y^{\wa}p(y)dy,
\eeq
so that by \eqref{splitDNalpha}, \eqref{Ialpha} and  \eqref{IINalpha}:
\beq
\label{DNalI-II}
D_N(\alpha)=\int_1^{2\pi N}y^{\wa}p(y)dy=I^{\alpha}+II^{\alpha}_N.
\eeq
Then,  by the definition \eqref{e.2.1} of $p$ in $[2\pi k,2\pi (k+1))$:
\[
\begin{aligned}
\int_{2\pi k}^{2\pi(k+1)}y^{\wa}p(y)dy=&
\int_{2\pi k}^{2\pi (k+1)}y^{\al}\Big[\frac{\pi^2}{6}-\frac{\pi}{2}(y-2\pi k)+\frac14(y-2\pi k)^2\Big]dy\\
=&
\int_{2\pi k}^{2\pi (k+1)} 
\Big[\frac{\pi^2}{6}+\frac{\pi}{2}(2\pi k)+\frac14(2\pi k)^2\Big] y^{\al}
-
\Big[\frac{\pi}{2}+\frac12(2\pi k)\Big]y^{\al+1}
+
\frac14 y^{\al+2}
dy\\
=&
\Big[\frac{\pi^2}{6}+\frac{\pi}{2}(2\pi k)+\frac14(2\pi k)^2\Big]\frac{1}{\al+1}[(2\pi (k+1))^{\al+1}-(2\pi k)^{\al+1}]\\
&-
\Big[\frac{\pi}{2}+\frac12(2\pi k)\Big]\frac{1}{\al+2}[(2\pi (k+1))^{\al+2}-(2\pi k)^{\al+2}]\\
&+
\frac14 \frac{1}{\al+3} [(2\pi (k+1))^{\al+3}-(2\pi k)^{\al+3}].
\end{aligned}
\]
By the definition \eqref{IINalpha} of $II^\al_N$ and rearranging the terms in the latter equality we then get
\[
\begin{aligned}
II^{\alpha}_N=\sum_{k=1}^{N-1}
\bigg[ & \frac{\pi^2}{6} \frac{1}{\al+1} [(2\pi (k+1))^{\al+1}-(2\pi k)^{\al+1}]
 -\frac{\pi}{2} \frac{1}{\al+2} [(2\pi (k+1))^{\al+2}-(2\pi k)^{\al+2}]\\
+&\frac14  \frac{1}{\al+3}  [(2\pi (k+1))^{\al+3}-(2\pi k)^{\al+3}]\\
+&\frac{\pi}{2}(2\pi k) \frac{1}{\al+1} [(2\pi (k+1))^{\al+1}-(2\pi k)^{\al+1}]
-\frac12(2\pi k) \frac{1}{\al+2} [(2\pi (k+1))^{\al+2}-(2\pi k)^{\al+2}]\\
+&\frac14(2\pi k)^2\frac{1}{\al+1}[(2\pi (k+1))^{\al+1}-(2\pi k)^{\al+1}]
\bigg].
\end{aligned}
\]
We use the equalities
\[
\sum_{k=1}^{N-1}[(k+1)^{\al+j}-k^{\al+j}]=
\sum_{k=1}^{N-1}(k+1)^{\al+j}-\sum_{k=1}^{N-1}k^{\al+j}=
N^{\al+j}-1\qquad j=1,2,3,
\]
and obtain
\beq
\label{temple}
\begin{aligned}
II^{\alpha}_N=&
\frac{\pi^2}{6} \frac{(2\pi)^{\al+1}}{\al+1} [N^{\al+1}-1]
 -\frac{\pi}{2} \frac{(2\pi)^{\al+2}}{\al+2} [N^{\al+2}-1]
+\frac14  \frac{(2\pi)^{\al+3}}{\al+3}  [N^{\al+3}-1]\\
&+\sum_{k=1}^{N-1}
\bigg[\frac{\pi}{2}(2\pi k) \frac{1}{\al+1} [(2\pi (k+1))^{\al+1}-(2\pi k)^{\al+1}]
-\frac12(2\pi k) \frac{1}{\al+2} [(2\pi (k+1))^{\al+2}-(2\pi k)^{\al+2}]\\
&+\frac14(2\pi k)^2\frac{1}{\al+1}[(2\pi (k+1))^{\al+1}-(2\pi k)^{\al+1}]
\bigg].
\end{aligned}
\eeq
We notice that 
\beq
\label{sumk}
\begin{aligned}
\sum_{k=1}^{N-1}
k [(k+1)^{\al+j}-k^{\al+j}]=&
\sum_{k=1}^{N-1} k (k+1)^{\al+j}
-\sum_{k=1}^{N-1} k k^{\al+j}\\
=&
-\sum_{k=1}^{N-1}(k+1)^{\al+j}
+\sum_{k=1}^{N-1}(k+1)^{\al+j+1}
-\sum_{k=1}^{N-1} k^{\al+j+1}\\
=&
1-\zeta_N(-\al-j)+N^{\al+j+1}-1
=-\zeta_N(-\al-j)+N^{\al+j+1}\,,
\end{aligned}
\qquad j=1,2
\eeq
where $\zeta_N(-\al-j)$ is the function defined in \eqref{zetaN}, and we used the equality $\sum_{k=1}^{N-1}(k+1)^{\al+j}=\sum_{m=2}^{N}m^{\al+j}=-1+\zeta_N(-\al-j)$. We also notice the equality
\beq
\label{sumk2}
\begin{aligned}
\sum_{k=1}^{N-1}
k^2 [(k+1)^{\al+1}-k^{\al+1}]=&
\sum_{k=1}^{N-1} (k+1)^{\al+1}
-2\sum_{k=1}^{N-1} (k+1)^{\al+2}
+ \sum_{k=1}^{N-1}  (k+1)^{\al+3} -\sum_{k=1}^{N-1} k^{\al+3}\\
=&-1+\zeta_N(-\al-1)+2-2\zeta_N(-\al-2)+N^{\al+3}-1\\
=&\zeta_N(-\al-1)-2\zeta_N(-\al-2)+N^{\al+3}.
\end{aligned}
\eeq
Using equalities \eqref{sumk} and \eqref{sumk2} in \eqref{temple} we have
\beq
\label{IIalN}
\begin{aligned}
II^{\alpha}_N=&
\frac{\pi^2}{6} \frac{(2\pi)^{\al+1}}{\al+1} [N^{\al+1}-1]
 -\frac{\pi}{2} \frac{(2\pi)^{\al+2}}{\al+2} [N^{\al+2}-1]
+\frac14  \frac{(2\pi)^{\al+3}}{\al+3}  [N^{\al+3}-1]\\
&+
\frac{\pi}{2} \frac{(2\pi)^{\al+2}}{\al+1}  [-\zeta_N(-\al-1)+N^{\al+2}]
-\frac12  \frac{(2\pi)^{\al+3}}{\al+2} [-\zeta_N(-\al-2)+N^{\al+3}]\\
&+\frac14 \frac{(2\pi)^{\al+3}}{\al+1}[\zeta_N(-\al-1)-2\zeta_N(-\al-2)+N^{\al+3}]\\
=&
-\frac{\pi^2}{6} \frac{(2\pi)^{\al+1}}{\al+1} +
\frac{\pi}{2} \frac{(2\pi)^{\al+2}}{\al+2} -\frac14  \frac{(2\pi)^{\al+3}}{\al+3}
+\frac{(2\pi)^{\al+3}}2 \Big[ \frac{1}{\al+2}-\frac{1}{\al+1}\Big]
\zeta_N(-\al-2)\\
&+\frac{\pi^2}{6} \frac{(2\pi)^{\al+1}}{\al+1} N^{\al+1}+
(2\pi)^{\al+2}\frac{\pi}{2} \Big[\frac{1}{\al+1}  -\frac{1}{\al+2}\Big]N^{\al+2}\\
&+(2\pi)^{\al+3}\frac12\Big[  \frac{1}{2(\al+3)} 
-\frac{1}{\al+2} +\frac{1}{2(\al+1)}\Big]N^{\al+3}.
\end{aligned}
\eeq
Using the latter equality and the formula \eqref{Ialpha-2} for $I^\al$  in \eqref{DNalI-II} we get
\[
\begin{aligned}
D_N(\alpha)=&
-\frac{\pi^2}{6}\frac1{\wa+1}
+\frac{\pi}{2}\frac1{\wa+2}
-\frac{1}{4}\frac1{\wa+3}\\
&
+\frac{(2\pi)^{\al+3}}2 \Big[ \frac{1}{\al+2}-\frac{1}{\al+1}\Big]
\zeta_N(-\al-2)\\
&+\frac{\pi^2}{6} \frac{(2\pi)^{\al+1}}{\al+1} N^{\al+1}+
(2\pi)^{\al+2}\frac{\pi}{2} \Big[\frac{1}{\al+1}  -\frac{1}{\al+2}\Big]N^{\al+2}\\
&+(2\pi)^{\al+3}\frac12\Big[  \frac{1}{2(\al+3)} 
-\frac{1}{\al+2} +\frac{1}{2(\al+1)}\Big]N^{\al+3},
\end{aligned}
\]
which is  the equality \eqref{res-DNal}.
\end{proof}

We observe that \eqref{res-DNal} can be written as 
\[
\begin{aligned}
D_N(\alpha)=&-\frac{\pi^2}{6}\frac1{\alpha+1}
+\frac{\pi}{2}\frac1{\alpha+2}
-\frac14\frac1{\alpha+3}-
\frac{(2\pi)^{\al+3}}{2(\al+2)(\al+1)} A_N(-2-\al)
+\frac{\pi^2}{6}\frac{(2\pi N)^{\al+1}}{\al+1}
\end{aligned}
\]
with
\beq
\label{ANs}
\begin{aligned}
A_N(s)
\equiv&
\zeta_{N}(s)
-\frac{1}{2}\frac{1}{N^s}
-\frac{1}{1-s}\frac{1}{N^{s-1}}\\
=&\sum_{n=1}^N\frac{1}{n^s}
-\frac{N^{1-s}}{1-s}
-\frac{N^{-s}}{2}
\end{aligned}
\eeq
This will be used in the proof of the following:
%%%%%
%LEMMA
%%%%%
\begin{lem}
\label{lem-limDNal}
Let $D_N(\alpha)$ be defined as in Lemma \ref{lem-DNal}, then for $\Re \al<-2$ 
\[
\lim_{N\to+\infty}D_N(\al)=
-\frac{\pi^2}{6}\frac1{\alpha+1}
+\frac{\pi}{2}\frac1{\alpha+2}
-\frac14\frac1{\alpha+3}-\frac{(2\pi)^{\al+3}}{2(\al+2)(\al+1)}\zeta(-2-\al).
\] 
\end{lem}
%%%%%
%PROOF
%%%%%
\begin{proof}
From \cite[Th. 4.11]{tit}
\beq
\label{tit411}
\zeta(s)=\sum_{n=1}^N\frac{1}{n^s}-\frac{N^{1-s}}{1-s}+O(N^{-\Re s})
=\zeta_N(s)-\frac{N^{1-s}}{1-s}+O(N^{-\Re s})
\eeq
uniformly for $\Re s>0$, $|\Im s|<2\pi N/C$, where $C$ is a given constant
greater than $1$, and $\zeta_N(s)$ was defined in \eqref{zetaN}. Then from
\eqref{tit411} and  the definition  \eqref{ANs} of $A_N(s)$  we have for such
values of $s$:
\[
A_N(s)
=\zeta(s) -\frac12 \frac{1}{N^s}- O(N^{-\Re s})\,.
\]
This implies that  for $\Re s>0$ we have  $\lim_{N\to\infty}A_N(s)= \zeta(s)$. Then $\Re \al<-2$ implies $\lim_{N\to\infty}A_N(-2-\al)= \zeta(-2-\al)$ and by the definition of $D_N(\al)$, for $\Re \al<-2$
\[
\begin{aligned}
&\lim_{N\to\infty}
D_N(\al)\\
&=
\lim_{N\to\infty}
\Big[
-\frac{\pi^2}{6}\frac1{\alpha+1}
+\frac{\pi}{2}\frac1{\alpha+2}
-\frac14\frac1{\alpha+3}-
\frac{(2\pi)^{\al+3}}{2(\al+2)(\al+1)} A_N(-2-\al)
+\frac{\pi^2}{6}\frac{(2\pi N)^{\al+1}}{\al+1}
\Big]\\
&=
-\frac{\pi^2}{6}\frac1{\alpha+1}
+\frac{\pi}{2}\frac1{\alpha+2}
-\frac14\frac1{\alpha+3}-\frac{(2\pi)^{\al+3}}{2(\al+2)(\al+1)}\zeta(-2-\al).
\end{aligned}
\]
\end{proof}

%%%%%%%%
%COROLLARY
%%%%%%%%
\begin{cor}
\label{cor-Dal}
Let $D(\alpha)$ be as in \eqref{e.2.7}, then for $\Real\alpha<-2$:
\beq
\label{eq-cor-Dal}
D(\alpha)=
-\frac{\pi^2}{6}\frac1{\alpha+1}
+\frac{\pi}{2}\frac1{\alpha+2}
-\frac14\frac1{\alpha+3}-\frac{(2\pi)^{\al+3}}{2(\al+2)(\al+1)}\zeta(-2-\al).
\eeq
\end{cor}
%%%%%
%PROOF
%%%%%
\begin{proof}
This is immediate from Lemma \ref{lem-limDNal} and $D(\al)=\lim_{N\to\infty}D_N(\al) $ for $\Re\al<-2$.
\end{proof}

\begin{rem}
\label{r.2.4}
$1)$
We also  can express $D(\alpha)$ as a Mellin transform. We first recall 
that the Mellin transform of a complex-valued continuous function 
$f$ on the positive real line is defined by
\beq
\label{mellin}
M(f)(s)\equiv\int_0^{+\infty}x^{s-1}f(x)dx, \ s\in\C
\eeq
(see, e.g., \cite{Ma06,ober}). It exists as an absolutely convergent integral if
$x\to x^{s-1}f(x)$ is in $L^1([0,+\infty))$.
Set
\[
p_0(y)\equiv
\left\{
        \begin{aligned}
            &p(y)      ,\quad  && \text{ for } y\in[1,+\infty),\\
            &0      ,  && \text{ for } y\in[0,1),
        \end{aligned}\right.
\]
with $p$ as in Section \ref{sec2} (e.g. \eqref{e.2.2}).
Then using the definition of $D$ in \eqref{e.2.7} we see that
\[
D(\alpha)=M(p_0)(\alpha+1),
\] 
for all $\alpha\in\C$ with $\Real\alpha<-1$.
In this way $D(\alpha)$ appears as the Mellin transform of the function $p_0$ evaluated
at $\alpha+1$.\\

$2)$
Let us derive an expression of $D$ as a convergent series containing incomplete gamma functions.
We have namely from \eqref{e.2.7} and \eqref{e.2.2}, for $\Real\alpha<-1$:
\begin{equation}
\label{e.2.24}
\begin{aligned}
D(\alpha)=&\int_1^{+\infty}y^{\alpha}\sum_{n=1}^{+\infty}\frac{\cos ny}{n^2}dy
=\sum_{n=1}^{+\infty}\frac{1}{n^2}\int_1^{+\infty}y^{\alpha}\cos ny dy\\
=&\frac{1}{2}\sum_{n=1}^{+\infty}\frac{i^{-\alpha-1}}{n^{3+\al}}[(-1)^{-\alpha-1}\Gamma(\alpha+1,-in)
+\Gamma(\alpha+1,in)],
\end{aligned}
\end{equation}
where we used dominated convergence to interchange sum and integral (since
$\left|\sum_{n=1}^{N}\frac{\cos ny}{n^2}\right|\leqslant\sum_{n=1}^{+\infty}\frac1{n^2}=\frac{\pi^2}{6}$
for all $N\in\N$) and 
\[
\int_1^{+\infty}x^{\nu}e^{ixy}dx
=(-iy)^{-\nu-1}\Gamma(\nu+1,-iy),
\]
for $\Real\nu<0$, cf. \cite[3.4, p. 199]{ober},
\[
\Gamma(\lambda,z)=\int_{z}^{+\infty}t^{\lambda-1}e^{-t}dt
=z^{\frac12\lambda-\frac12}e^{-\frac12z}W_{\frac\lambda2-\frac12,\frac12\lambda}(z)
\]
(\cite[p.  255]{ober}, $W_{k,\nu}(z)$ being a Whittaker function,
\cite[p.  254]{ober}). All the series in formula \eqref{e.2.24} are absolutely
convergent for $\Re\alpha <-1$.
\end{rem}

%%%%%%%%%%%%%%%%%%%%%%%%%%%%%%%%%%%%
%SECTION
%%%%%%%%%%%%%%%%%%%%%%%%%%%%%%%%%%%%
\section{The representation of $\zeta$ in terms of the associated Clausen function $p$.}
\label{sec3}

Integral representations of $\zeta$ are known since the original work \cite{ri}.
In particular we mention the one given by Riemann
\[
\zeta(s)=\Gamma\Big(\frac{s}{2}\Big)^{-1}\pi^{\frac{s}{2}}
\left\{\int_1^{\infty}
[x^{\frac{s}{2}}+x^{\frac{1-s}{2}}]\psi(x)dx-\frac1{s(1-s)}\right\}
\]
\[
\psi(x)\equiv \sum_{n=1}^{\infty}e^{-n^2\pi x}
\]
(see \cite{ri}, \cite[p. 16]{Edw74} and  \cite{tit}). Another
representation  we
would like to   mention  is
\[
\zeta(s)=\frac{\Gamma(-s)}{2\pi i}
\int_{\Gamma}
\frac{(-x)^s}{e^x-1}\frac{dx}{x},
\]
$\Gamma$ being a contour which begins at $+\infty$, descends the real axis, circles 
the singularity at the origin once in the positive direction, and returns up to the positive real axis to $+\infty$, and where $(-x)^s=\exp[s\log(-x)]$ is defined in the usual way for $-x$ not on the negative real axis (see, e.g., \cite[p. 137]{Edw74}). Other somewhat similar representations are given by the Riemann Siegel integral formula (see, e.g., \cite[p. 166]{Edw74}). Also the remainder in the approximation formula for $\zeta(s)$ has an integral form:
\[
\zeta(s)=\zeta_N(s)-\frac{N^{1-s}}{1-s}-s\int_N^{\infty}\left\{t\right\}t^{-s-1}dt,
\]
where $\{t\}$ denotes the fractional part of $t$ and  $\Real s>0$ (see, e.g., \cite[p. 144]{Ten}). Further integral
representations are given, e.g., by \cite[pp. 102, 103, Eqs. 41-43, for all
$s\in\C$; Eqs. 44-46, for $0<\Re s<1$; Eqs. 48, 50,  for $\Re s>1$;
Eqs. 49 for $\Re s>0$]{srich}.
Another representation for $\Re s>1$ is given in
\cite[p. 405]{RSdmj8196} and for $\Re s>0$ in \cite{ivi}, see also
\cite[p. 14]{tit}.  

The main aim of this section is to state and prove  Theorem \ref{t.3.1} 
which gives a new integral representation of $\zeta$ for $\Real s>0$. This representation is  in terms of the dilogarithm function,
via the function $D$ (introduced in \eqref{e.2.7} and discussed in Section
2). We shall also deduce some consequences from our representation. 

%%%%%%%
%THEOREM
%%%%%%%
\begin{thm}
\label{t.3.1}
For any $\Real s>0$ we have 
\begin{align}\label{e.3.1}
\zeta(s)=\frac{ 2s(1+s)}{(2\pi)^{1-s}}
\left\{
\frac{\pi^2}{6}\frac1{1+s}-\frac{\pi}{2}\frac1s
-\frac14 \frac1{1-s}-D(-2-s)\right\}
\end{align} 
where $D$ is as in Section \ref{sec2} (formula \eqref{e.2.7}).
\end{thm}
%%%%%
%PROOF
%%%%%
\begin{proof}
From Corollary \ref{cor-Dal} we have, for $\Real\alpha<-2$:
\[
D(\alpha)=-\frac{\pi^2}{6}\frac1{\alpha+1}+\frac\pi2\frac1{\alpha+2}
-\frac14\frac1{\alpha+3}-\frac{(2\pi)^{\al+3}}{2(\alpha+1)(\alpha+2)}\zeta(-\alpha-2).
\]
Setting $s=-\alpha-2$ we have for $\Real s>0$:
\[
\begin{aligned}
D(-2-s)=&-\frac{\pi^2}{6}\frac1{-s-1}+\frac\pi2\frac1{-s}
-\frac14\frac{1}{-s+1}-\frac{(2\pi)^{1-s}}{2(-s-1)(-s)}\zeta(s)\\
=&
\frac{\pi^2}{6}\frac1{s+1}-\frac\pi2\frac1{s}
-\frac14\frac{1}{1-s}-\frac{(2\pi)^{1-s}}{2(s+1)s}\zeta(s),
\end{aligned}
\]
from which the theorem follows by noticing that $(2\pi)^{1-s}\neq0$.
\end{proof}

Let us derive some consequences from  Theorem \ref{t.3.1} (Remarks
\ref{r.3.1}, \ref{r.3.2} and \ref{r.3.3}).
\begin{rem}
\label{r.3.1}
The representation in Theorem \ref{t.3.1} gives a new proof of the fact that $\zeta(s)$ is 
a meromorphic function in $\Real s>0$, with a simple pole at $s=0$. In fact $s\to D(-2-s)$ is  holomorphic in $\Real s>0$, since it is equal to $M(p_0)(-1-s)$, which is holomorphic in 
$\Real s>0$ (and even in $\Real s>-1$), being the Mellin transform at $-1-s$ of the function $p_0$ and 
$x\to x^{-2-u} p_0(x)$ is integrable on $\R_+$ (cf., e.g., \cite{tit}, ch. I, 1.29, pp. 46-47).
The other functions entering Theorem \ref{t.3.1} are meromorphic, with a simple pole at $s=1$.
\end{rem}

\begin{rem}
\label{r.3.2}
If $s_0$ is a zero of $\zeta$, with $\Real s_0>0$, then
\begin{align}
\label{e.3.2}
D(-2-s_0)=\frac{\pi^2}{6}\frac1{1+s_0}-\frac\pi2\frac1{s_0}
-\frac14\frac1{1-s_0}
\end{align}
(since $s_0\neq0$ and also $1+s_0\neq0$, because $\Real(1+s_0)=1+\Real s_0>1$). In particular 
if $s_0=u_0+iv_0, u_0>0, \ v_0\in\R$, then
\begin{align}
\begin{split}
\label{Astern}
D(-2-s_0)=\frac{\pi^2}{6}\frac1{1+u_0+iv_0}-\frac{\pi}{2}
\frac1{u_0+iv_0}-\frac14\frac1{1-u_0-iv_0}\\
=\frac{\pi^2}{6}\frac{1+u_0-iv_0}{(1+u_0)^2+v_0^2}-\frac\pi2
\frac{u_0-iv_0}{u_0^2+v_0^2}
-\frac14\frac{1-u_0-iv_0}{(1-u_0)^2+v_0^2}.
\end {split}
\end{align}
In particular
\begin{align}\label{e.3.3a}
\Real D(-2-s_0)=\frac{\pi^2}{6}\frac{1+u_0}{(1+u_0)^2+v_0^2}
-\frac\pi2\frac{u_0}{u_0^2+v_0^2}
-\frac14\frac{1-u_0}{(1-u_0)^2+v_0^2}
\end{align}
\begin{align}\label{e.3.3b}
\im D(-2-s_0)=v_0\left\{
-\frac{\pi^2}{6}\frac1{(1+u_0)^2+v_0^2}
+\frac\pi2\frac1{u_0^2+v_0^2}-\frac14\frac1{(1-u_0)^2+v_0^2}
\right\}.
\end{align}
For the special case $u_0=\frac12$ we get then
\begin{align}
\label{e.3.3}
D\Big(-\frac52-iv_0\Big)=\frac{\pi^2}{6}\frac1{\frac32+iv_0}-\frac\pi2\frac1{\frac12+iv_0}
-\frac14\frac1{\frac12-iv_0}.
\end{align}
From this it follows:
\begin{align}
\label{e.3.4}
\Real D\Big(-\frac52-iv_0\Big)=\frac12
\left\{
\frac{\pi^2}{6}\frac3{(\frac32)^2+v_0^2}-\frac\pi2\frac1{(\frac12)^2+v_0^2}
-\frac14\frac1{(\frac12)^2+v_0^2}
\right\}
\end{align}
\begin{align}
\label{e.3.5}
\im D\Big(-\frac52-iv_0\Big)=v_0
\left\{
-\frac{\pi^2}{6}\frac1{(\frac32)^2+v_0^2}+\frac\pi2\frac1{(\frac12)^2+v_0^2}
-\frac14\frac1{(\frac12)^2+v_0^2}
\right\}.
\end{align}
These formulae can be used for the numerical verification whether $s_0$ is a zero
of the $\zeta$ on the critical line (this obviously relies
on an efficient evaluation of the integral 
\[
D\Big(-\frac52-iv_0\Big)=\int_1^{+\infty}y^{-\frac52-iv_0}p(y)dy,
\]
with $p$ defined in Section \ref{sec2}).
Possibly also results on the distribution of zeros of $\zeta$ on the
critical line can be obtained in this way. Numerical work in this direction is planned. 
\end{rem}

One could ask the question whether it could be possible to exclude that there exists a sequence  $s_0(n)=u_0+iv_0(n)$ with $1>u_0>0$, $u_0\neq\frac12$, $v_0(n)\to+\infty$ satisfying \eqref{Astern} for all $n$, this implying that there are at most finitely many zeros of $\zeta$ on $\Real s=u_0$. For $u_0=\frac12$ one knows, however, that there is such a sequence, by Hardy's proof of the existence of infinitely many zeros of $\zeta$ on the critical line, see, e.g., \cite{tit} (ch. 11). A comparison of the first and second order in an asymptotic expansion in powers of $\frac1v$ for $v\to\infty$ of the two members of \eqref{Astern} gives a negative answer to this question, in the sense that it does not permit to distinguish between the behavior at $u_0=\frac12$ and at $u_0\neq \frac12$ (up to these orders in $\frac1{v}$).

%%%%%%%%%
%REMARK
%%%%%%%%%
\begin{rem}
\label{r.3.3}
Let us briefly indicate how one can deduce in a simple way a zero-free region
for $\zeta$ close to the real line, using the integral representation of $\zeta$ given in  Theorem \ref{t.3.1}.
\end{rem}

%%%%%
%PROOF
%%%%%

From Theorem \ref{t.3.1} if $s_0=u_0+iv_0$ satisfies $u_0>0$ and $\zeta(s_0)=0$ then
\begin{align}
\label{e.3.6}
\frac{2s_0(1+s_0)}{(2\pi)^{1-s_0}}\left\{\frac{\pi^2}{6}\frac1{1+s_0}-\frac\pi2\frac1{s_0}-\frac14\frac1{1-s_0}
-D(-2-s_0)\right\}=0.
\end{align}
Since $s_0\neq0,\ 1+s_0\neq0$ this implies, if $v_0\neq0$, dividing by $2v_0 s_0(1+s_0)/(2\pi)^{1-s_0}$:
\begin{align}\label{e.3.7}
\frac1{v_0}\Big[\frac{\pi^2}{6}\frac1{1+s_0}-\frac\pi2\frac1{s_0}-\frac14\frac1{1-s_0}\Big]
=\frac1{v_0}D(-2-s_0).
\end{align}
But
\begin{equation}
\begin{aligned}
\label{e.3.8}
\im\Big(\frac{D(-2-s)}{v}\Big)
=&-\int_1^{+\infty}x^{-2-u}\frac{\sin(v\ln x)}{v}p(x)dx\\
=&-\int_1^{+\infty}x^{-2-u}\frac{\sin(v\ln x)}{v\ln x}(\ln x)p(x)dx.
\end{aligned}
\end{equation}
For any $u>0$ define 
\begin{align}\label{e.3.9a}
c(u)\equiv -\int_1^{+\infty}x^{-2-u}(\ln x)p(x)dx\,,\qquad u>0.
\end{align} 
Then we have
\begin{align}
\label{e.3.11}
\begin{split}
\left|\im \frac{D(-2-s)}{v}-c(u)\right|
=&\left|\int_1^{+\infty}y^{-2-u}\Big[\frac{\sin(v\ln y)}{v\ln y}-1\Big](\ln y) \, p(y)dy\right|\\
\leqslant&\int_1^{+\infty}y^{-2-u}
\left|\frac{\sin(v\ln y)-v\ln y}{v\ln y}\right|(\ln y)|p(y)|dy.
\end{split}
\end{align}
But from Taylor's formula
\beq
\label{3.13a}
\left|\frac{\sin(v\ln y)-v(\ln y)}{v\ln y}\right|
\leqslant\frac{v^{\ast3}(y)(\ln y)^3}{3!v\ln y}
\leqslant\frac{v^2(\ln y)^2}{3!}
\eeq
for all $v>0$ for some $v^{\ast}(y)\in[0,v],\ y\in[1,+\infty)$.
On the other hand the left hand side of \eqref{3.13a} is also bounded by 2.  Introducing these bounds into \eqref{e.3.11} we get
\beq
\label{e.3.12}
\left|\im\frac{D(-2-s)}{v}-c(u)\right|
\leqslant\frac{v^2}{3!}\int_1^{+\infty}y^{-2-u}(\ln y)^3|p(y)| dy
\leqslant b(u,v),
\eeq
with
\[
b(u,v)\equiv
\frac{\pi^2}3\min\left\{\frac{v^2}{3!}\frac12 \int_1^{+\infty}y^{-2-u}(\log y)^3 dy, \int_1^{+\infty} y^{-2-u} (\log y) dy\right\}
\]
and
where we used $|p(y)|\leqslant{\pi^2}/{6},\ y\in[1,+\infty), u>0$. 

Assume that $s_0=u_0+iv_0$ satisfies $u_0>0$ and $\zeta(s_0)=0$ and set 
\[
B(u_0,v_0)\equiv\frac1{v_0}\Big[\frac{\pi^2}{6}\frac1{1+s_0}-\frac\pi2\frac1{s_0}-\frac14\frac1{1-s_0}\Big].
\]
By Eq. \eqref{e.3.7} one has that $B(u_0,v_0)=\frac{D(-2-s_0)}{v_0}$, then the inequality \eqref{e.3.12} implies that
\begin{align}
\label{e.3.14}
c(u_0)-b(u_0,v_0)\leqslant\im B(u_0,v_0)
\leqslant c(u_0)+b(u_0,v_0).
\end{align}
But
\beq
\label{imB}
\begin{aligned}
\im B(u_0,v_0)=-\frac{\pi^2}{6}\frac1{(1+u_0)^2+v_0^2}+\frac\pi2\frac1{u_0^2+v_0^2}-\frac14\frac1{(1-u_0)^2+v_0^2}\,,
\end{aligned}
\eeq
for all $0<u_0<1$.

The inequality \eqref{e.3.14} is certainly not satisfied, i.e., $s_0=u_0+iv_0$ is not a zero of the function $\zeta(s)$, if
\beq
\label{estB}
\im B(u_0,v_0)>c(u_0)+b(u_0,v_0)
\eeq
or
\beq
\label{e.3.23b}
\im B(u_0,v_0)<c(u_0)-b(u_0,v_0).
\eeq
But $\max_{u_0\in(0,1)} b(u_0,v_0) \leq \frac{\pi^2}3\min\left\{\frac{v_0^2}{2},1\right\}$, where we used that for any $\alpha<-1$:
\[
\int_1^{+\infty}y^{\alpha}(\ln y)^3 dy
=\frac{6}{(\alpha+1)^4}\;,\qquad
\int_1^{+\infty}y^{\alpha}(\ln y) dy
=\frac{1}{(\alpha+1)^2}
\]
(as seen by integrations by parts).
From \eqref{imB} it follows that for $u_0\in(0,1)$ \eqref{estB} is certainly satisfied for values $(u_0,v_0)$ such that
\begin{align}
\label{s.12a}
\frac{\pi^2}{2}\frac1{u_0^2+v_0^2}
>\max_{u_0 \in(0,1)}c(u_0)
+ \frac{\pi^2}3\min\left\{\frac{v_0^2}{2},1\right\}+\frac{\pi^2}{6}\frac{1}{(1+u_0)^2+v_0^2}+\frac14\frac1{(1-u_0)^2+v_0^2}.
\end{align}

To check the inequality \eqref{s.12a} we need a lower bound for  $\max_{u_0 \in(0,1)}c(u_0)$, this can be obtained numerically as follows:
%%%%%%%%%%%%%%%%%
%NEW PART 1
%%%%%%%%%%%%%%%%%
first we
rewrite $c(u_0)$ as  
\begin{align}
c(u_0) = & -\int_1^{2\pi N}x^{-2-u_0}(\ln x)p(x)dx -\int_{2\pi N}^{+\infty}x^{-2-u_0}(\ln x)p(x)dx
\nonumber
\\
\equiv&c_{N}(u_0) + r_{N}(u_0)\,,
\label{e:bro-1}
\end{align}
where $N\in\N$. Then we  notice that $p(x)\geq0$ for $x\in[2\pi n, 2\pi
n+\pi\frac{3-\sqrt3}{3}] \cup [2\pi n+\pi\frac{3+\sqrt3}{3}, 2\pi n+2\pi]$,
for any $n=0,1,2, . . . $, and  $p(x)\leqslant 0$ for $x\in[2\pi n+\pi\frac{3-\sqrt3}{3},2\pi n +\pi\frac{3+\sqrt3}{3}]$, for any $n=0,1,2, . . . $, and we rewrite $c_{N}(u_0)$ as the sum of its positive and negative part, resp. $c_{N,+}(u_0)$ and $c_{N,-}(u_0)$:
\begin{equation}
\label{e:hersh}
c_{N}(u_0) =  c_{N,+}(u_0) - c_{N,-}(u_0) 
\end{equation}
Since $c_{N,+}(u_0)$ and  $c_{N,-}(u_0)$ are  decreasing functions of
$u_0$, for any $N\in \NN$ and $0<u_0<1$, the equality \eqref{e:hersh}
gives  
\begin{equation}
\label{e:ref-1}
c_{N,+}(1) -  
 c_{N,-}(0) 
\leqslant  c_{N}(u_0) \leqslant c_{N,+}(0) -  
c_{N,-}(1) \,.
\end{equation}
To get a lower and an upper bound  for $r_{N}(u_0)$ for $0<u_0<1$ we notice
that for any $x>0$, we have $-\frac{\pi^2}{12}\leqslant p(x)\leqslant \frac{\pi^2}{6}$. Then we get
\begin{equation}
\label{e:ref-3}
m_{N} \equiv -\frac{\pi^2}{6} \int_{2\pi N}^{+\infty}x^{-2}(\ln x)dx
\leqslant r_{N}(u_0)\leqslant
\frac{\pi^2}{12}\int_{2\pi N}^{+\infty}x^{-2}(\ln x)dx \equiv M_N\,.
\end{equation}
From the formula  \eqref{e:bro-1}  and from the bounds \eqref{e:ref-1} and  \eqref{e:ref-3}  it follows that 
\[
c_{N,+}(1) -  
 c_{N,-}(0) + m_{N}
\leqslant  c(u_0) \leqslant c_{N,+}(0) -  
c_{N,-}(1) + M_{N}\,.
\]
The terms in the upper and lower bound can be easily computed numerically. E.g. for $N=100$ we
obtain 
\begin{equation}
\label{e:c0num}
-0.12 \lesssim  c(u_0) \lesssim 0.36\,.
\end{equation}

Now using the numerical estimate given in \eqref{e:c0num} we see
that \eqref{s.12a} holds for values of $u_0>0$ near $0$ and values of $v_0 \leq 1.1$
(e.g. for $u_0=0.1,\ v_0=1.1$ we have $\sim 4.05$ at the left hand side and $\sim
3.15$ at the right hand side of the latter inequality).
Let us remark that the bound \eqref{e.3.23b} does not seem to yield better results. 
As it stands the zero-free region is smaller than the one obtained in \cite{albce} 
(by another method, see also this reference for other zero-free regions).
Obviously even within the limit of what can be reached from Theorem
\ref{t.3.1} our considerations are not optimal in many directions. E.g. one
could restrict
$u_0$ to smaller intervals, such as $u_0\in(\frac12,1)$, to get slightly  better bounds
for $c(u_0)$. But in order to obtain really stronger results one would require better numerical or analytical techniques to handle the function $D$.

%%%%%%%%%%%%%%%%%%%%%%%%%%%%%%%%%%%%
%SECTION
%%%%%%%%%%%%%%%%%%%%%%%%%%%%%%%%%%%%
\section{Other representations of $\zeta$ in terms of periodic functions}
\label{sec4}

Set
\beq
\label{tildep}
\tilde p(x)\equiv p(x)+\frac{\pi^2}{12}, \ x\in\R.
\eeq
Since $\min p=p(\pi)=-\pi^2/12$ the function $\tilde p$ satisfies $\tilde p(x)\geq0$ for all $x\in\R$. From the definition \eqref{e.2.1} of $p$, $\tilde p(x)$ can be written as 
\[
\tilde p(x) = \frac{1}{4}\big(x-(2n+1)\pi\big)^2\quad x\in[2\pi n,2\pi (n+1)),\,n\in\N_0\equiv\left\{0\right\}\cup\N\,.
\]

For $\Re\al <-1$, set  
\beq
\label{widetildeD}
\widetilde D(\alpha)\equiv\int_1^{+\infty}y^{\alpha}\tilde p(y)dy=D(\al) - \frac{\pi^2}{12} \frac{1}{\al+1}.
\eeq
%%%%%%
%REMARK
%%%%%%
\begin{rem} From Corollary \ref{cor-Dal} and  Theorem \ref{t.3.1}, and from the definition \eqref{widetildeD} of $\widetilde 
D$, it follows that for any $\Re \al<-2$ and  $\Re s>0$ we have
\[
\widetilde D(\alpha)=
-\frac{\pi^2}{4}\frac1{\alpha+1}
+\frac{\pi}{2}\frac1{\alpha+2}
-\frac14\frac1{\alpha+3}-\frac{(2\pi)^{\al+3}}{2(\al+2)(\al+1)}\zeta(-2-\al),
\]
hence the representation of $\zeta$ given in Theorem \ref{t.3.1} can also be written in the form:

\beq
\label{zetaD}
\zeta(s)=
\frac{ 2s(1+s)}{(2\pi)^{1-s}}
\bigg\{
\frac{\pi^2}{4}\frac1{1+s}-\frac{\pi}{2}\frac1s
-\frac14 \frac1{1-s}-\widetilde D(-2-s)\bigg\}, \Real s>0.
\eeq

\end{rem}

We define the function 
\beq
\label{tildeq}
 q(x)\equiv \frac{1}{2}\big(x-(2n+1)\pi\big)\quad x\in[2\pi n,2\pi (n+1)),\,n\in\N_0
\eeq
We observe that $ q(x)\equiv  \frac{d}{dx} \tilde p(x)=\frac{d}{dx}p(x)$ for any $x\in (2\pi n, 2\pi (n+1))$ and $n\in\N_0$.

For any $\Re \al<-1$ set 
\beq
\label{tildeEal}
 E(\al)\equiv 
\int_1^{+\infty}y^{\alpha} q(y)dy.
\eeq

%%%%%%%%%
%PROPOSITION
%%%%%%%%%
\begin{prop}
\label{prop-tildeE}
Let $ E(\alpha)$ be as in \eqref{tildeEal}, then for $\Real\alpha<-1$:
\beq
\label{tildeE}
 E(\alpha)=
\frac{\pi}{2}\frac{1}{\al+1}-\frac{1}{2}\frac{1}{\al+2}+\frac{(2\pi)^{2+\al}}{2(\al+1)}\zeta(-1-\al).
\eeq
\end{prop}
%%%%%
%PROOF
%%%%%
\begin{proof}
The proof  follows closely what was done to obtain the formula \eqref{eq-cor-Dal} in Corollary \ref{cor-Dal}. For any integer $N>1$, set 
\beq
\label{tildeEN}
\begin{aligned}
 E_N(\alpha)\equiv&\int_1^{2\pi N}y^{\alpha} q(y)dy=
\int_1^{2\pi}y^{\alpha} q(y)dy+
\sum_{k=1}^{N-1}\int_{2\pi k}^{2\pi (k+1)}y^{\alpha} q(y)dy\\
=&\frac12\int_1^{2\pi}y^{\alpha}(y-\pi)dy+
\sum_{k=1}^{N-1}\frac12\int_{2\pi k}^{2\pi (k+1)}y^{\alpha}(y-(2k+1)\pi)dy\\
=&\frac12\frac1{\al+2}[(2\pi)^{\al+2}-1]-\frac\pi2\frac{1}{\al+1}[(2\pi)^{\al+1}-1]\\
&+\sum_{k=1}^{N-1}\bigg\{\frac12\frac{1}{\al+2}[(2\pi (k+1))^{\al+2}-(2\pi k)^{\al+2}]
-\frac\pi2\frac{(2k+1)}{\al+1}[(2\pi (k+1))^{\al+1}-(2\pi k)^{\al+1}]
\bigg\}\\
=&\frac12\frac1{\al+2}[(2\pi N)^{\al+2}-1]-\frac\pi2\frac{1}{\al+1}[(2\pi)^{\al+1}-1]-\sum_{k=1}^{N-1}\frac\pi2\frac{(2k+1)}{\al+1}[(2\pi (k+1))^{\al+1}-(2\pi k)^{\al+1}].
\end{aligned}
\eeq
Since
\[
\begin{aligned}
&\sum_{k=1}^{N-1}\frac\pi2\frac{(2k+1)}{\al+1}[(2\pi (k+1))^{\al+1}-(2\pi k)^{\al+1}]
=\frac\pi2\frac{(2\pi)^{\al+1}}{\al+1}\sum_{k=1}^{N-1}(2k+1)[(k+1)^{\al+1}-k^{\al+1}]\\
=&\frac\pi2\frac{(2\pi)^{\al+1}}{\al+1}\bigg[
2\sum_{k=1}^{N-1}(k+1)^{\al+2}-
\sum_{k=1}^{N-1}(k+1)^{\al+1}
-2\sum_{k=1}^{N-1}k^{\al+2} -\sum_{k=1}^{N-1}k^{\al+1}\bigg]\\
=&\frac\pi2\frac{(2\pi)^{\al+1}}{\al+1}\bigg[
2[N^{\al+2}-1]-
\sum_{k=2}^{N}k^{\al+1}
-\sum_{k=1}^{N}k^{\al+1}+N^{\al+1}
\bigg]
\\
=&\frac\pi2\frac{(2\pi)^{\al+1}}{\al+1}\bigg[
2N^{\al+2}+N^{\al+1}
-1-2\sum_{k=1}^{N}k^{\al+1}\bigg].
\end{aligned}
\]
We use the formula \eqref{tit411}, see, e.g., \cite{tit}, as we did in the proof of Lemma \ref{lem-limDNal}, and obtain
\[
\begin{aligned}
&\sum_{k=1}^{N-1}\frac\pi2\frac{(2k+1)}{\al+1}[(2\pi (k+1))^{\al+1}-(2\pi k)^{\al+1}]\\
=&\frac\pi2\frac{(2\pi)^{\al+1}}{\al+1}\bigg[
2N^{\al+2}+N^{\al+1}
-1-2\zeta(-\al-1)-2\frac{N^{\al+2}}{\al+2}+O(N^{\Re (\al+1)})\bigg]\\
=&\frac\pi2\frac{(2\pi)^{\al+1}}{\al+1}\bigg[
2\frac{\al+1}{\al+2}N^{\al+2}
-1-2\zeta(-\al-1)+O(N^{\Re (\al+1)})\bigg],
\end{aligned}
\]
which holds true uniformly for $\Re (\al)<-1$, $|\Im \al|<2\pi N/C$, where $C$ is a given constant greater than $1$.

We use the latter equality in Eq. \eqref{tildeEN} for $ E_N(\al)$ and obtain
\[
 E_N(\alpha)=
-\frac12\frac1{\al+2}+\frac\pi2\frac{1}{\al+1}
+\frac12\frac{(2\pi)^{\al+2}}{\al+1}
\zeta(-\al-1)+O(N^{\Re (\al+1)}).
\]
The statement follows from
\[
 E(\al)=
\lim_{N\to\infty} E_N(\al)=
-\frac12\frac1{\al+2}+\frac\pi2\frac{1}{\al+1}
+\frac12\frac{(2\pi)^{\al+2}}{\al+1}
\zeta(-\al-1)
\]
for all $\al<-1$.
\end{proof}
%%%%%%
%COROLLARY
%%%%%%
\begin{cor}
The following integral representation for $\zeta$ in terms of the function $ E$ defined by \eqref{tildeq}, \eqref{tildeEal}  holds, for all $\Real s>0$: 
\beq
\label{zetaE}
\zeta(s)=\frac{2s}{(2\pi)^{1-s}}
\bigg[-
\frac{\pi}{2}\frac{1}{s}-\frac{1}{2}\frac{1}{1-s}
- E(-1-s)
\bigg]
\eeq
\end{cor}
\begin{proof}
This is an immediate consequence of Proposition \ref{prop-tildeE}.
\end{proof}

\begin{rem}\label{rem.4.4}
One could use \eqref{zetaE} to deduce  explicit zero-free regions for
$\zeta(s)$, similarly as  in Remark \ref{r.3.3}.
Moreover from \eqref{zetaE}
and 
 $| q(y)|\leqslant \pi/2$ we get
\[
\big| E(-1-u-iv)\big|=\bigg|\int_1^{+\infty}y^{-1-u-iv} q(y)dy\bigg|
\leqslant \int_1^{+\infty}y^{-1-u}| q(y)|dy\leqslant
\frac{\pi}{2}\frac1{u}.
\]
From this and \eqref{zetaE} we deduce  the following simple explicit bound, for all $\Real s>0$:
\label{r.4.3}
\beq
\label{zetaE-2}
\begin{aligned}
|\zeta(u+iv)|\leqslant&
\frac{1}{(2\pi)^{1-u}}
\bigg[\pi+\frac{|u+iv|}{|1-u-iv|}
+2|u+iv|\big| E(-1-u-iv)\big|\bigg]\\
\leqslant&
\frac{1}{(2\pi)^{1-u}}
\bigg[\pi+\frac{u+|v|}{|v|}
+\pi\frac{(u+|v|)}{u}
\bigg]=\frac{1}{(2\pi)^{1-u}}
\bigg[\frac\pi{u}|v|+2\pi+1+\frac{u}{|v|}
\bigg]
\end{aligned}
\eeq

For $u\geqslant\frac12$ our  bound is of the type  given in \cite[2.12.2]{tit} (namely $\zeta(s)=O(|v|)$, for $u\geqslant \frac12$).
To improve our bound one should exploit the oscillatory nature of the integrand, which would require a separate analysis. 

For other bounds on $|\zeta(s)|$ see, e.g.,  \cite{Edw74} (p.184), \cite{KV92}
(pp. 116-118, 125), \cite{jam} (pp.  38,104,106,113,196),
\cite{tit} (p. 113).

\end{rem}

%%%%%%
%REMARK
%%%%%%
\begin{rem}
By comparison of Eqs. \eqref{zetaD} and \eqref{zetaE} we have
\[
\widetilde D(-2-s)=\frac{1}{1+s}\bigg[\frac14(1-\pi)^2+ E(-1-s)\bigg]
\]
for any $\Re s>0$. The same result can be obtained directly from the definition \eqref{widetildeD} of $\widetilde D(\al)$, by splitting the integral  into the domains $[2\pi k,2\pi (k+1))$ and integrating  by parts. 
\end{rem}
Whereas the representation of $\zeta(s)$ given by Theorem \ref{t.3.1} involves $D(-2-s)$, which is built with the function $p(x)$ which is quadratic in the fundamental domain $[0,2\pi)$, the one given by Remark \ref{r.4.3} involves $ E(-1-s)$, which is built with the function $q(x)$ which is linear in the fundamental domain $[0,2\pi)$ (being the derivative of $p$). The next considerations  will involve a function $f(x)$ which is constant in the fundamental domain. 

Set 
\[
f(x)\equiv (-1)^n\qquad x\in[2\pi n,2\pi (n+1)),\ n\in\N_0
\]
and for all $\Re\al<-1$
\beq
\label{F}
F(\al)\equiv \int_1^{+\infty} y^\al f(y)dy .
\eeq
Then we have
%%%%%%%%%
%PROPOSITION
%%%%%%%%%
\begin{prop}
\label{prop-F}
Let $F(\alpha)$ be as in \eqref{F}, then for $\Real\alpha<-1$:
\beq
F(\al)=-\frac1{\al+1}+
2\frac{(2\pi)^{\al+1}}{\al+1}
(1-2^{2+\al})\zeta(-\al-1).
\eeq
\end{prop}
%%%%%
%PROOF
%%%%%
\begin{proof}
For any integer $N\in\N$, set 
\beq
\label{FN}
\begin{aligned}
F_N(\alpha)\equiv&\int_1^{2\pi N}y^{\alpha}f(y)dy=
\int_1^{2\pi}y^{\alpha}dy+
\sum_{k=1}^{N-1}(-1)^k\int_{2\pi k}^{2\pi (k+1)}y^{\alpha}dy\\
=&\frac1{\al+1}[(2\pi)^{\al+1}-1]+\sum_{k=1}^{N-1}(-1)^k\frac{1}{\al+1}[(2\pi (k+1))^{\al+1}-(2\pi k)^{\al+1}].
\end{aligned}
\eeq
We notice that
\[
\begin{aligned}
&\sum_{k=1}^{N-1}(-1)^k\frac{1}{\al+1}[(2\pi (k+1))^{\al+1}-(2\pi k)^{\al+1}]
=\frac{(2\pi)^{\al+1}}{\al+1}\bigg[\sum_{k=2}^{N}(-1)^{k-1}k^{\al+1}-
\sum_{k=1}^{N-1}(-1)^k k^{\al+1}\bigg]\\
=&\frac{(2\pi)^{\al+1}}{\al+1}\bigg[-1+\sum_{k=1}^{N}(-1)^{k-1}k^{\al+1}+
\sum_{k=1}^{N-1}(-1)^{k-1} k^{\al+1}\bigg]\\
=&
\frac{(2\pi)^{\al+1}}{\al+1}\bigg[-1+2\sum_{k=1}^{N}(-1)^{k-1}k^{\al+1}
-(-1)^{N-1}N^{\al+1}\bigg].
\end{aligned}
\]
Using the latter equality in Eq. \eqref{FN} for $F_N(\al)$ we obtain
\[
\begin{aligned}
F_N(\alpha)=&
\frac1{\al+1}[(2\pi)^{\al+1}-1]+
\frac{(2\pi)^{\al+1}}{\al+1}\bigg[-1+2\sum_{k=1}^{N}(-1)^{k-1}k^{\al+1}
-(-1)^{N-1}N^{\al+1}\bigg]\\
=&-\frac1{\al+1}+
\frac{(2\pi)^{\al+1}}{\al+1}\bigg[2\sum_{k=1}^{N}(-1)^{k-1}k^{\al+1}
-(-1)^{N-1}N^{\al+1}\bigg].
\end{aligned}
\]
Since for all $\Re \al<-1$ the series on the right hand side of the latter equation converges and $\lim_{N\to\infty}F_N(\al)=F(\al)$, then 
\[
F(\al)=\lim_{N\to\infty}
F_N(\al)=-\frac1{\al+1}+
2\frac{(2\pi)^{\al+1}}{\al+1}\sum_{k=1}^{\infty}(-1)^{k-1}k^{\al+1}
\]
for any $\Re \al<-1$. The statement follows from
\beq
\label{e.4.11a}
\sum_{k=1}^\infty(-1)^{k-1}k^{\al+1}
=(1-2^{2+\al})\zeta(-\al-1)
\eeq
for any $\Re\al<-1$, see \cite[Sec. 2.2]{tit}.
\end{proof}
%%%%%%
%REMARK
%%%%%%
\begin{cor}\label{cor.4.7}
From Proposition \ref{prop-F} we get the integral representation for the $\zeta$-function in terms of $F$: 
\beq
\zeta(s)
=\frac{1}2\,\frac{(2\pi)^{s}}{1-2^{1-s}}-\frac{s}2\,\frac{(2\pi)^{s}}{1-2^{1-s}}F(-1-s),
\eeq
for any $\Re s>0$. 
\end{cor}

\begin{rem}
This representation yields  the simple bound 
\[
\begin{aligned}
|\zeta(u+iv)|\leqslant&
\frac{(2\pi)^u}{2}\frac{1}{\sqrt{1+2^{2-2u}-2^{2-u}\cos(v\ln2)}}\Big[1+|s||F(-1-s)|\Big]\\
\leqslant&\frac{(2\pi)^u}{2}\frac{1}{\sqrt{1+2^{2-2u}-2^{2-u}\cos(v\ln2)}}\Big[1+\frac{|s|}{u}\Big],
\end{aligned}
\]
where in the latter step we used $|F(-1-s)|\leqslant \int_1^{+\infty}y^{-1-u}dy=\frac{1}{u}$. 
The growth in $|s|$ of the bound is similar to the one obtained in Remark \ref{r.4.3}.
To sharpen the bound one should again exploit the oscillations of the integrand in $F(-1-s)$.

We also remark that using \eqref{e.4.11a} we get from Corollary \ref{cor.4.7}, for $\Real s>0$:
\[
\sum_{n=1}^{+\infty}\frac{(-1)^{n-1}}{n^s}=\frac{s}{2}(2\pi)^s\bigg[\frac{1}{s}-F(-1-s)\bigg],
\]
which is the known representation of the $\zeta$-function as a convergent alternating series, see, e.g., \cite{tit} Sect. 2.2.
\end{rem}

%%%%%%%%%%%%%%%%%%%%%%%%%%%%%%%%%%%%
%SECTION
%%%%%%%%%%%%%%%%%%%%%%%%%%%%%%%%%%%%
\section{Relations with the M\"untz formula} 
\label{sec5}

M\"untz's formula (see \cite{lew81} and\cite{tit} (p. 29)) relates the Mellin transform of a modified theta transform of a test function $f$ with the product of the Mellin transform of the test function and the $\zeta$-function.
This formula is presented here in a generalized form, assuming only that $f$ and its Fourier transform $\mf(f)$ are in $L^1(\R)$, see next Proposition \ref{s2p1}.
Further in this section we present an $s\to 1-s$ symmetrized version (under a stronger condition on $f$ and $\mf(f)$) of this generalized M\"untz formula (Proposition \ref{s2p3}). We then relate M\"untz formula for a certain choice $f=\phi$ of $f$ to our integral representation for the $\zeta$-function given by Theorem \ref{t.3.1} (this is the content of Remark \ref{r.5.11})

\begin{prop}[Generalized M\"untz formula]
\label{s2p1}
Let $f \in L^1(\R)$, and assume that $\mf(f)\in L^1(\R)$, $\mf(f)(x)\equiv\int_{\R}e^{2\pi ixy}f(y)dy$ being the Fourier transform of $f$. Then:
\begin{itemize}
\item[1)] \begin{align*}
                \Theta_N(f)(x)\equiv\sum_{n=1}^{N}f(nx)
            \end{align*}
            converges absolutely for all $x>0$ to
            \begin{align*}
                \Theta(f)(x)\equiv\sum_{n=1}^{\infty}f(nx).
            \end{align*}
\item[2)] For any $0<\operatorname{Re}(s)<1$ the following statements hold:
\begin{itemize}
            \item[a)] the Mellin transform of $f$
                        \begin{align*}
                            M(f)(s)\equiv\int_0^{\infty}x^{s-1}f(x)dx
                        \end{align*}
                       exists in the sense of Lebesgue integrals;
\item[b)] define $\check{\Theta}(f)$ by
                        \begin{align*}
                            \check{\Theta}(f)(x)\equiv\Theta(f)(x)-
%\frac{\mf(f)(0)}{2x},\ 
\frac{1}{x}\int_0^\infty f(y)dy,\ 
x>0,
                        \end{align*}
                        then
                        the Mellin transform of $\check{\Theta}(f)$ exists
                        in the sense of Lebesgue integrals and one has
                        \begin{align*}
                            M(f)(s)\zeta(s)=M(\check{\Theta}(f))(s).
                        \end{align*}
\end{itemize}
\end{itemize}
\end{prop}

\begin{proof}
This is a consequence of a theorem of M\"untz \cite{mun} (see also \cite[pp. 28 - 29, 2.11]{tit}),
combined with the absolute convergence of the integrals and almost sure
 convergence of the series in $\check{\Theta}(f)(x)$  being made clear
 in the work by Burnol \cite[Sect. 3.2 Prop. 3.15]{bur05}, to which we also refer for more details. Note that the assumptions are weaker than in \cite{mun} but are covered by the result of \cite{bur05}. 
\end{proof}

\begin{rem}
\label{rem22}
M\"untz formula has been analyzed, extended and
exploited in very interesting recent work by J.-F. Burnol, see, e.g.
\cite{bur, burn3, bur03, bur05, burn2}, and B\'aez-Duarte, see, e.g., \cite{baez}. It has also been exploited (independently of above
work) for the study of zero-free regions of $\zeta$ in \cite{albce}.
In algebraic contexts similar formulae appear in pioneering work by
Tate \cite{tat} and, more recently, in work by Connes
\cite{conn} (see also, e.g. \cite{cm07a}) and Meyer \cite{mey}.
\end{rem}

We shall now rewrite M\"untz's formula in a way which exploits better the intrinsic
basic symmetry with respect to $s\to 1-s$ used classically for deriving
the functional equation (which one obtains in the particular case where
$f(x)=e^{-\pi x^2}$). For this one uses Poisson's summation formula:

\begin{prop}[Poisson summation formula]
\label{s2p2} Let $f\in L^1(\R)$  and $\mf(f)$  (defined as in Prop. \ref{s2p1}) be
continuous and assume they satisfy
\[
|f(x)|+|\mf(f)(x)|\leqslant\frac{c}{(1+|x|)^{1+\delta}}
\]
for some $c,\delta>0$, and all $x\in\R$. Then for any $a, b>0$ such that $ab = 2\pi$ we have
\[
\sqrt{a} \sum_{k \in \Z} f(ak)= \sqrt{b} \sum_{k \in
\Z} \widehat{f}(bk),
\]
with $\widehat{f}(y)\equiv \frac{1}{\sqrt{2\pi}}\int_{\R}
f(x)e^{-iyx}dx$. Both series are absolutely convergent and converge uniformly to a continuous
function of $a$ resp. $b$.
\end{prop}
\begin{proof}
$f\in L^1(\R)$ implies that $\mf(f)$ exist. The rest follows from
e.g., \cite[Corollary 2.6, (2.8), with $f(x)$ replaced by
$\sqrt{a}f(ax)$]{stwei}  see also, e.g., \cite{gra,pi},  
 \cite[Eqs. 13-14, p. 70]{zyg}.
\end{proof}

\begin{cor}
\label{s2c1}  Let $f$ be as in Proposition \ref{s2p2}. Then
\[
\mf(f)(y) = \int_{\R} e^{2\pi
ixy}f(x)dx=\sqrt{2\pi}\hat{f}(-2\pi y), \qquad y\in \R, 
\]
and for any $a>0$ we have
\[ a \sum_{k \in \Z}f(ak)=
\sum_{k\in \Z} \mf(f)\left(-\frac{k}{a}\right)
\]
(with both series being absolutely convergent to continuous functions of $a$).
\end{cor}
\begin{proof}
This is immediate from Proposition \ref{s2p2} with $b={2\pi}/{a}$ and from the fact that
\[
\widehat{f}(by)=\frac{1}{\sqrt{2\pi}} \mf(f)\Big(-\frac{by}{2\pi}\Big) 
\] 
by the definitions of $\widehat{f}$ and $ \mf(f)$.
\end{proof}

\begin{cor}
\label{s2c2} Let $f$ be as in Proposition \ref{s2p2} and assume that
$y \mapsto f(y)$ is an even function ($y \in \R$). Then for $x>0$
\[\Theta(f)(x)=\frac{1}{2}\left[\frac{\mf(f)(0)}{x}-f(0)\right]+
\frac{1}{x}\Theta(\mf(f))\left(\frac1{x}\right)
\]
\end{cor}
\begin{proof}
We first remark that $f$ being even implies $\mf(f)$ being even; on
the other hand, from Corollary \ref{s2c1} we have
\[
 a \sum_{k \in \N}f(ak) + a f(0) + a\sum_{k \in \N}
f(-ak) = \sum_{k \in \N} \mf(f)\left(-\frac{k}{a}\right)+ \mf(f)(0)+
\sum_{k \in \N}\mf(f)\left(\frac{k}{a}\right).
\]
Both $f$ and $\mf(f)$ are even, therefore we get
\[
2a \sum_{k \in \N}f(ak) + a f(0) = 2\sum_{k \in \N}
\mf(f)\left(\frac{k}{a}\right)+ \mf(f)(0).
\]
From this, replacing $a$ by $x$, and using the definition of
$\Theta$ we get
\[
2x \Theta(f)(x)+x
f(0)=2\Theta(\mf(f))\left(\frac{1}{x}\right)+\mf(f)(0), \ x>0,
\] which proves Corollary \ref{s2c2}.
\end{proof}

\begin{prop}[Symmetrized version of the generalized M\"untz formula]
\label{s2p3}
Let $f$ be even and  as in Proposition \ref{s2p2}. Then for any $0<\operatorname{Re}(s)<1$:
\begin{align}
\label{g4}
  M(f)(s) \zeta(s)= \frac{1}{2}\left[\frac{\mf(f)(0)}{s-1}-\frac{f(0)}{s}\right]+I(f)(s),
\end{align}
with
\begin{align}
\label{g4-1}
I(f)(s) \equiv \int_{1}^{\infty} x^{s-1} \Theta(f)(x)dx +
               \int_{1}^{\infty} x^{-s} \Theta(\mf(f))(x)dx.
\end{align}
Both integrals on the right hand side of this formula for $I(f)$
exist in Lebesgue's sense.
\end{prop}
\begin{proof}
From Proposition \ref{s2p1} we have for any $0<a<1<b<+\infty$:
\begin{align}
\label{g5}
M(f)(s)\zeta(s)=\lim_{a\downarrow 0} \int_{a}^{1} x^{s-1}\check{\Theta}(f)(x)dx
                  + \lim_{b \uparrow \infty} \int_{1}^{b} x^{s-1}\check{\Theta}(f)(x)dx.
\end{align}
Both limits exist and the integrals are in Lebesgue's sense. From the definition of
$\check{\Theta}$ (in Prop. \ref{s2p1}) and Corollary \ref{s2c2}, then
\begin{align}
\begin{split}
\label{g6}
\int_{a}^{1}x^{s-1}\check{\Theta}(f)(x)dx
=& \frac{1}{2}\int_{a}^{1}x^{s-1}\left[\frac{\mf(f)(0)}{x}-f(0)\right]dx\\
    &+\int_{a}^{1}x^{s-1}x^{-1}\Theta(\mf(f))\left(\frac1{x}\right)dx
    -\int_a^{1}x^{s-1}\frac{\mf(f)(0)}{2x}dx\\
=&-\frac{f(0)}{2}\int_a^{1}x^{s-1}dx
        +\int_a^{1}x^{s-2}\Theta(\mf(f))\left(\frac1{x}\right)dx\\
=&-\frac{f(0)}{2s}[1-a^s]+\int_a^{1}x^{s-2}\Theta(\mf(f))\left(\frac1{x}\right)dx.
\end{split}
\end{align}

In the latter expression the first term converges for $a \downarrow
0$ (for all $0<\operatorname{Re}(s)<1$) to $ - {f(0)}/(2s)$.
Since the limit of the left hand side of \eqref{g6} for $a
\downarrow 0$ exists, also absolutely, by Proposition \ref{s2p1},
\[\lim_{a\downarrow
0}\int_{a}^{1}x^{s-2}\Theta(\mf(f))\left(\frac1{x}\right)dx
\]
must also exist (for $0<\operatorname{Re}(s)<1$). Thus, under our
assumption on $s$:
\begin{align}
\label{g7}
\lim_{a \downarrow 0} \int_{a}^{1}x^{s-1}\check{\Theta}(f)(x)dx =
  -\frac{1}{2}f(0)\frac{1}{s}+
  \lim_{a \downarrow 0} \int_{a}^{1}x^{s-2}\Theta(\mf(f))\left(\frac1{x}\right)dx.
\end{align}
On the other hand, for any $0<a<1$ (by the change of variables $x\to x'=\frac1{x}$):
\beq
\label{g8}
\int_{a}^{1}x^{s-2}\Theta(\mf(f))\left(\frac1{x}\right)dx
  =\int_{\frac{1}{a}}^{1}x'^{2-s}\Theta\left(\mf(f)\right)(x')\left(-\frac{1}{x'^2}\right)dx' 
  =\int_{1}^{\frac{1}{a}}x^{-s}\Theta(\mf(f))(x)dx.
\eeq
From \eqref{g7} and \eqref{g8} we get
\beq
\label{g9}
\lim_{a \downarrow 0} \int_{a}^{1}x^{s-1}\check{\Theta}(f)(x)dx =
  -\frac1{2}f(0)\frac{1}{s}
  + \lim_{a \downarrow 0} \int_{1}^{\frac{1}{a}}x^{-s}\Theta(\mf(f))(x)dx,
\eeq
with both limits existing absolutely. By \eqref{g5} and \eqref{g9}
we get under our assumptions on $s$:
\[
M(f)(s)\zeta(s)=-\frac1{2}f(0)\frac{1}{s}
    +\lim_{a \downarrow 0} \int_{1}^{\frac{1}{a}}x^{-s}\Theta(\mf(f))(x)dx+\lim_{b \uparrow \infty}
    \int_{1}^{b}x^{s-1}\check{\Theta}(f)(x)dx.
\]
We note that from Proposition \ref{s2p1} the existence of the limit
of the latter integral is assured. But, from the definition of $\check{\Theta}$:
\beq
\label{gl26}
\begin{aligned}
\int_1^{b}x^{s-1}\check{\Theta}(f)(x)dx=&\int_1^{b}x^{s-1}\Theta(f)(x)dx
        -\int_1^{b}x^{s-1}\frac{\mf(f)(0)}{2x}dx\\
        =&\int_1^{b}x^{s-1}\Theta(f)(x)dx
        -\frac{\mf(f)(0)}{2\left(s-1\right)}(b^{s-1}-1).
\end{aligned}
\eeq
Since the second term in the latter expression converges for
$b\to+\infty$ to $\frac{\mf(f)(0)}{2\left(s-1\right)}$, also the
first term must converge, as $b\to+\infty$.  Hence we get
\beq
\label{g10}
M(f)(s)\zeta(s)=\frac{1}{2}\left(\frac{\mf(f)(0)}{s-1}-\frac{f(0)}{s}\right)
    +\int_{1}^{\infty}x^{-s}\Theta(\mf(f))(x)dx+\int_{1}^{\infty}x^{s-1}\Theta(f)(x)dx,
\eeq
which proves \eqref{g4} (with the integrals converging absolutely).
\end{proof}

%%%%%%
%REMARK
%%%%%%
\begin{rem}
We note that in Proposition \ref{s2p1} only the values of $f$ on $\RE_+$ are used. Corollary \ref{s2c2} (under the stronger assumption on $f$ given in Proposition \ref{s2p2}) serves to replace the integral over $\RE_+$ on the right hand side of Proposition \ref{s2p1} by integrals on $[1,+\infty)$, and here an even extension of a given $f$ on $\RE_+$ to $\RE$ is used in order to deduce Proposition  \ref{s2p3}.  
\end{rem}

%%%%%%%%%
%PROPOSITION
%%%%%%%%%
\begin{prop}
\label{prop5.8}
Let $f$ be even and as in Proposition \ref{s2p2}. Then for any $0<\Re s<1$
\begin{itemize}
\item[1)]
\beq
\label{rel-1}
I(\mf(f))(s)=I(f)(1-s)
\eeq
\item[2)]
\beq
\label{rel-2}
M(\mf(f))(s) \zeta(s)
=M(f)(1-s) \zeta(1-s)
\eeq
\end{itemize}
where $I(f)$ was defined in \eqref{g4-1}.
\end{prop}
%%%%%
%PROOF
%%%%%
\begin{proof}
We first remark that by the assumptions in  Proposition \ref{s2p2}, we have that both $f$ and $\mf(f)$ are in $L^1(\R)$. Moreover $I(f)(s)$ and $I(\mf(f))(s)$  exist in Lebesgue's sense for any $0<\Re s<1$. Moreover being $f$ even, then $\mf(f)$ is also even. We also have,   $\mf(\mf(f))(x)= f(-x),x \in \R$ (as   follows easily from the definition of $\mf$ and $f$, $\mf(f)\in L^1(\R)$, see, e.g. \cite[p. 173]{weav2}). But $f$ being even, we then get  $\mf(\mf(f))=f$. Then by \eqref{g4-1}  one has 
%%%
\[
\begin{aligned}
I(\mf(f))(s) = & \int_{1}^{\infty} x^{s-1} \Theta(\mf(f))(x)dx +
               \int_{1}^{\infty} x^{-s} \Theta(\mf(\mf(f)))(x)dx\\
 =&\int_{1}^{\infty} x^{s-1} \Theta(\mf(f))(x)dx +
               \int_{1}^{\infty} x^{-s} \Theta(f)(x)dx
               = I(f)(1-s)               
\end{aligned}
\]
%%%
which proves the formula \eqref{rel-1}.  Moreover by  Proposition
 \ref{s2p3}, $M(\mf(f))(s)\zeta(s)$ can be written  by using Eq.   \eqref{g4}.    Formula  \eqref{rel-2} is a trivial consequence of \eqref{rel-1} and  of the fact that $\mf(\mf(f))(0) = f(0)$ implies 
\[
\frac{1}{2}\left(\frac{\mf(\mf(f))(0)}{s-1}-\frac{\mf(f)(0)}{s}\right)
=\frac{1}{2}\left(\frac{f(0)}{s-1}-\frac{\mf(f)(0)}{s}\right).
\]
\end{proof}

%%%%%%
%REMARK
%%%%%%
\begin{rem}
\label{r5.9}
From  formula \eqref{rel-2} and from the functional equation
 $\Gamma(z/2)\pi^{-\f
 z2}\zeta(z)=\Gamma((1-z)/2)\pi^{\f{z-1}2}\zeta(1-z)$, for all
 $z\in\CO$ (see, e.g., \cite{tit}, \cite[9.535]{GraRyz07} and \cite{oss0809}) it follows that, for any $f$ satisfying the assumptions of Proposition \ref{prop5.8}, the Mellin transform $M(\mf(f))$ of $\mf(f)$ exists and is given in terms of the Mellin transform of $f$  by
%%%
\beq
M(\mf(f))(s)=
\frac{\zeta(1-s)}{\zeta(s)}M(f)(1-s)
=\frac{\Gamma(s/2)\pi^{1-s}}{\Gamma((1-s)/2)}M(f)(1-s),
\qquad 0<\Re s<1,\ \zeta(s)\neq0.
\eeq
%%%
\end{rem}
Let us point out that whereas the right hand side of the first equality is well defined only if $\zeta(s)\neq0$, 
the right hand side of the second equality is also well defined for all $0<\Real s<1$, so that
\beq\label{e.5.12a}
M(\mf(f))(s)=\frac{\Gamma(\frac{s}{2})\pi^{1-s}}{\Gamma(\frac{1-s}{2})}
M(f)(1-s),
\eeq
for all $0<\Real s<1$.

In the following we consider a particular choice $\phi$ of  the test function $f$ in the  M\"untz formula.
This will permit to obtain explicit formulae for both $\mf(f)$ and $M(\mf(f))$
and, more importantly, to relate our integral representation for $\zeta$ with
M\"untz formula (in the form of \eqref{prop-5.10-eq2} below, see also Remark \ref{r.5.11}).
%%%%%%%%%
%PROPOSITION
%%%%%%%%%
\begin{prop}
\label{p3.1}
Let $\phi(x)\equiv 1-|x|$ for $ x\in[-1,+1],\ \phi(x)\equiv0,\ x\in(-\infty,-1)\cup(1,+\infty)$.
Then:
\begin{itemize}
\item[1)]
\[
\mf(\phi)(y)=\int_{\R}e^{2\pi ixy}\phi(x)dx
=\frac{1-\cos(2\pi y)}{2\pi^2y^2},\ \ y\in\R\backslash\left\{0\right\}
\]
\[
\mf(\phi)(0)=1,\ \ \text{ for } y=0 
\]
\item[2)] $\phi$ has all properties of the function $f$ in Proposition \ref{prop5.8}.
\item[3)] For all $0<\Re s$ the Mellin transform of $\phi$ is given by
\beq
\label{prop-5.10-eq1}
M(\phi)(s)=\frac{1}{s(s+1)}.
\eeq
For $0<\Real s<1$ one has
\beq
\label{prop-5.10-eq2}
M(\mf(\phi))(s)\zeta(s)=\frac{1}{(1-s)(2-s)}\zeta(1-s).
\eeq
For $s$ such that $\zeta(s)\neq0$,
$0<\Real s<1$ we also have
\[
M(\mf(\phi))(s)=\frac{\zeta(1-s)}{\zeta(s)(1-s)(2-s)}.
\]
\item[4)] 
The Mellin transform of $\mf(\phi)$ is given for all $0<\Real s<1$ by
\beq\label{e.5.15}
M(\mf(\phi))(s)=\frac{\Gamma(s/2)\pi^{1-s}}{\Gamma((1-s)/2)}\frac{1}{(1-s)(2-s)}.
\eeq
\end{itemize}
\end{prop}
%%%%%
%PROOF
%%%%%
\begin{proof}
\begin{itemize}
\item [\emph{1)}]This is a simple computation (in fact the formula can be extracted, e.g., from \cite[p. 186]{Don69}), minding that
	\[
	\mf(\phi)(y)=(2\pi)^{\frac12}\hat{\phi}(-2\pi y)
	\]
	and
	\[
	\hat{\phi}(z)=\sqrt{\frac2\pi}\frac{1-\cos z}{z^2},\ \ z\in\R.
	\]
\item[\emph{2)}] $\phi$ is even, continuous and obviously satisfies
	$|\phi(x)|\leqslant\frac{c}{(1+|x|)^{1+\delta}}$,
	for some $c>0, \delta>0$.
	Moreover $\mf(\phi)$ is also continuous, since it behaves as $1- \frac{(2\pi y)^4}{4! 2\pi^2y^2}$ in 
	a neighborhood of $0$ (as seen from Taylor's formula) and is bounded by $\frac2{2\pi^2y^2}$ for large $|y|$,
	hence satisfies $|\hat{\phi}(x)|\leqslant\frac{c}{(1+|x|)^{1+\delta}},$ for some $c,\delta>0$.
	Let us remark, in addition, that
	$\phi$ is in fact $C^{\infty}$ on 
	$\R\backslash(\left\{0\right\}\cup\left\{+1\right\}\cup\left\{-1\right\})$, with right and left 
	derivatives at $0$ and $\pm1$, with finite jumps from the left to the right.
	$\mf(\phi)$ belongs to $C^{\infty}(\R)$ and is even (as Fourier transform of the even function $\phi$).
\item[\emph{3)}] By the definitions of $M$, $\phi$ 
\[
\begin{aligned}
M(\phi)(s)=&\lim_{\ep\downarrow0}\int_{\ep}^{+\infty}x^{s-1}\phi(x)dx=
\lim_{\ep\downarrow0}\int_{\ep}^{1}x^{s-1}(1-x)dx\\
=&\lim_{\ep\downarrow0}\left[\left.\frac1s x^s\right|_{\ep}^1
-\left.\frac1{s+1} x^{s+1}\right|_{\ep}^1\right]
=\lim_{\ep\downarrow0}\left[\frac1s-\frac1s\ep^s-\frac1{s+1}+\frac1{s+1}\ep^{s+1}\right]\\
=&\frac1s-\frac1{s+1}=\frac1{s(s+1)}
\end{aligned}
\]
(where we used $\Real s>0$). Formula  \eqref{prop-5.10-eq2} comes from formula \eqref{rel-2} in Proposition \ref{prop5.8} and from \eqref{prop-5.10-eq1}.
\item[\emph{4)}] This is an immediate consequence of \eqref{e.5.12a}  (a consequence of the functional equation for $\zeta$) and \eqref{prop-5.10-eq1}.
\end{itemize}
\end{proof}

%%%%%%
%REMARK
%%%%%%
\begin{rem}
\label{r.5.11}
The computation of the Mellin transform in Proposition \ref{p3.1}, $\mathit{4}$), of the function $\mf(\phi)$ in Proposition \ref{p3.1}, $\mathit{1}$), also might have some interest in itself (this Mellin transform does not seem to be contained, e.g., in    \cite{Obe74}). 

Let us also point out that \eqref{prop-5.10-eq2} with \eqref{e.5.15} yield, in
turn, the well known functional equation for the Riemann zeta function.

Let us stress furthermore  that we can derive \eqref{prop-5.10-eq2} using our integral representation for $\zeta$ given in Theorem \ref{t.3.1}. From the definition of $\phi$ given in Proposition  \ref{p3.1} and by using Eq. \eqref{g4} one has namely
\beq
\label{start}
M(\mf(\phi))(s) \zeta(s)= \frac{1}{2}\frac{1}{s(s-1)}+\int_{1}^{\infty} x^{s-1} \Theta(\mf(\phi))(x)dx,
\eeq
where we used the definition \eqref{g4-1} of $I(\phi)$, the fact that $\mf(\mf(\phi))=\phi$, and the fact that $\phi(x)=0$ for $x>1$. From the computation of $\mf(\phi)$ in Proposition \ref{p3.1}, $1)$ we have, for $y\in[1,+\infty)$
\[
\Theta(\mf(\phi))(y)=\sum_{n=1}^{+\infty}\frac{1-\cos(2\pi ny)}{2\pi^2n^2y^2}
=\frac1{12}\frac1{y^2}-\frac1{2\pi^2}\frac1{y^2}p(2\pi y),
\]
where we used $\sum_{n=1}^{+\infty}\frac1{n^2}=\zeta(2)=\frac{\pi^2}{6}$ and the definition \eqref{e.2.2} of $p$. Then 
\beq
\label{end-0}
\begin{aligned}
\int_{1}^{\infty} x^{s-1} \Theta(\mf(\phi))(x)dx
=&\frac1{12}\int_1^{+\infty}x^{s-3}dx
-\frac1{2\pi^2}\int_1^{+\infty}x^{s-3} p(2\pi x)dx\\
=&-\frac{1}{12}\frac{1}{s-2}-\frac{1}{\pi(2\pi)^{s-1}}\int_1^{2\pi}y^{s-3}p(y)dy
+\frac{1}{\pi(2\pi)^{s-1}} D(s-3).
\end{aligned}
\eeq
From Corollary \ref{cor-Dal} one has, on the other hand:
\beq
\label{end-2}
\frac{1}{\pi(2\pi)^{s-1}} 
D(s-3)=
\frac{1}{\pi(2\pi)^{s-1}} \bigg[
-\frac{\pi^2}{6}\frac1{s-2}
+\frac{\pi}{2}\frac1{s-1}
-\frac14\frac1{s}-\frac{(2\pi)^{s}}{2(s-1)(s-2)}\zeta(1-s)\bigg].
\eeq
From the definition of $p$, on the other hand:
\beq
\label{end-3}
\begin{aligned}
\frac{1}{\pi(2\pi)^{s-1}}\int_1^{2\pi}y^{s-3}p(y)dy
=&\frac{1}{\pi(2\pi)^{s-1}}I^{s-3}\\
=&\frac{1}{\pi(2\pi)^{s-1}}\bigg[
\frac{\pi^2}{6}\frac1{s-2}((2\pi)^{s-2}-1)
-\frac{\pi}{2}\frac1{s-1}((2\pi)^{s-1}-1)
+\frac{1}{4}\frac1{s}((2\pi)^{s}-1)
\bigg]
\end{aligned}
\eeq
where $I^\al$ was defined in \eqref{Ialpha} and computed in  \eqref{Ialpha-2}. Using \eqref{end-2} and \eqref{end-3} in \eqref{end-0}, and then \eqref{end-0} in \eqref{start}, it follows that 
\[
M(\mf(\phi))(s) \zeta(s)=\frac{1}{(s-1)(s-2)}\zeta(1-s), 
\]
which is formula \eqref{prop-5.10-eq2}
\end{rem}

\begin{rem}
We also remark that computing in another way $M(\mf(\varphi))$ we can get an explicit integral, which might have some interest in itself. In fact using Proposition \ref{p3.1}, $\mathit{1}$), for $f=\varphi$, and the definition \eqref{mellin} of  $M$ we have 
\beq
\label{star}
\begin{aligned}
M(\mf(\phi))(s)=&
\int_0^1x^{s-1}\frac{1-\cos(2\pi x)}{2\pi^2x^2}dx
+\int_1^\infty x^{s-1}\frac{1-\cos(2\pi x)}{2\pi^2x^2}dx\\
=&
\int_0^1x^{s-1}\frac{1-\cos(2\pi x)}{2\pi^2x^2}dx
-\frac{1}{2\pi^2}\frac{1}{s-2}-\frac{1}{2\pi^2}\int_1^\infty x^{s-3}\cos(2\pi x)dx.
\end{aligned}
\eeq
Using, e.g., \cite[5.24]{Obe74} (for $a=2\pi$, $c=1$, $b=0$, $z=s-2$):
\[
\int_1^\infty x^{s-3}\cos(2\pi x)dx=
\frac{1}{2}(2\pi i )^{-(s-2)}\Gamma(s-2,2\pi i)+\frac{1}{2}(-2\pi i)^{-(s-2)}\Gamma(s-2,-2\pi i).
\]
Inserting this into \eqref{star} and comparing with \eqref{e.5.15} we get 
\beq
\begin{aligned}
M(\mf(\phi))(s)=
\frac{\Gamma(s/2)\pi^{1-s}}{\Gamma((1-s)/2)}\frac{1}{(1-s)(2-s)}
=&
\int_0^1x^{s-1}\frac{1-\cos(2\pi x)}{2\pi^2x^2}dx-
\frac{1}{2\pi^2}\frac{1}{s-2}\\
&-\frac{1}{4\pi^2}
\Big[(2\pi i )^{-(s-2)}\Gamma(s-2,2\pi i)+(-2\pi i)^{-(s-2)}\Gamma(s-2,-2\pi i)\Big].
\end{aligned}
\eeq
From this we can obtain an expression of the integral on the right hand side 
(``incomplete Mellin transform" of $\mf(\varphi)$)
in terms of incomplete gamma functions, namely for $0<\Re s<1$:
\beq
\begin{aligned}\label{e.5.22}
\int_0^1x^{s-1}\frac{1-\cos(2\pi x)}{2\pi^2x^2}dx=&
\frac{\Gamma(s/2)\pi^{1-s}}{\Gamma((1-s)/2)}\frac{1}{(1-s)(2-s)}
+
\frac{1}{2\pi^2}\frac{1}{s-2}\\
&-(2\pi i)^{-s}\Big[\Gamma(s-2,2\pi i)+\Gamma(s-2,-2\pi i)\Big].
\end{aligned}
\eeq
\end{rem}

Formula \eqref{e.5.22} is used in the following proposition to express a certain series containing factorials in terms of incomplete gamma functions.

\begin{prop}
The following summation formula holds for all $0<\Real s<1$:
\beq
\label{icing}
\begin{aligned}
\frac{1}{2\pi^2}\sum_{k=1}^\infty\frac{(2\pi)^{2k}}{(2k)!}\frac{(-1)^k}{s-2+2k}=&
\frac{\Gamma(s/2)\pi^{1-s}}{\Gamma((1-s)/2)}\frac{1}{(1-s)(2-s)}
+
\frac{1}{2\pi^2}\frac{1}{s-2}\\
&-(2\pi i)^{-s}\Big[\Gamma(s-2,2\pi i)+\Gamma(s-2,-2\pi i)\Big].
\end{aligned}
\eeq
\end{prop}

\begin{proof}
This follows from \eqref{e.5.22} observing that 
 the left hand side can be expressed in the following way (by inserting the power series expansion of $\cos(2\pi x)$ and using dominated convergence to exchange sum and integration): 
\begin{align}\label{e.5.22a}
\int_0^1x^{s-1}\frac{1-\cos(2\pi x)}{2\pi^2x^2}dx=\frac{1}{2\pi^2}\sum_{k=1}^\infty\frac{(2\pi)^{2k}}{(2k)!}\frac{(-1)^k}{s-2+2k}
\end{align}
Comparison of \eqref{e.5.22} and \eqref{e.5.22a} immediately yields the summation formula (which does not seem to appear in the usual tables on series, e.g., \cite{han}).
\end{proof}

%%%%%%%%%%%%%%%%%%%%%%%%%%%%%%%%%%%%%%%%%%%
%%%%%%%%%%%%%%%%%%%%%%%%%%%%%%%%%%%%%%%%%%%
%APPENDIX
%%%%%%%%%%%%%%%%%%%%%%%%%%%%%%%%%%%%%%%%%%%
%%%%%%%%%%%%%%%%%%%%%%%%%%%%%%%%%%%%%%%%%%%
\appendix

%%%%%%%%%%%%%%%%%%%%%%%%%%%%%%%%%%%%
%SECTION
%%%%%%%%%%%%%%%%%%%%%%%%%%%%%%%%%%%%
\section{Another derivation of the formula for $D(\wa)$} 
\label{sec6}
We sketch another derivation of the formula for $D(\alpha)$ which enters our integral representation for $\zeta$ given in Theorem \ref{t.3.1}.
We do not provide all the details but point out some explicit formulae which might have some interest in themselves.

\begin{lem}\label{l.2.1}
Let $D(\alpha)$ be defined by \eqref{e.2.7}, and assume
$\Real \alpha<-1$. Then:
\[
D(\alpha)=\frac{\pi^2}{6}\frac1{\alpha+1}[(2\pi)^{\alpha+1}-1]
-\frac{\pi}{2}\frac1{\alpha+2}[(2\pi)^{\alpha+2}-1]
+\frac14\frac1{\alpha+3}[(2\pi)^{\alpha+3}-1]+\widetilde{A}(\alpha),
\]
with
\begin{align*}
\widetilde{A}(\alpha)\equiv\sum_{k=1}^{+\infty}(2\pi k)^{\alpha}
\left\{
\frac{\pi^2}{6}(2\pi)\sum_{l=0}^{+\infty}{\alpha\choose l}
\frac1{l+1}\frac1{k^l}
-\frac{\pi}{2}(2\pi)^2\sum_{l=0}^{+\infty}{\alpha\choose l}
\frac1{l+2}\frac1{k^l}+\frac14(2\pi)^3\sum_{l=0}^{+\infty}{\alpha\choose l}\frac1{l+3}\frac1{k^l}
\right\}.
\end{align*}
\end{lem}
\begin{proof}
Splitting the integration domain $[1,+\infty)$ in \eqref{e.2.7} into
\[
[1,2\pi)\cup[2\pi,+\infty)=[1,2\pi)\cup\bigcup_{k=1}^{+\infty}[2\pi k,2\pi (k+1))
\]
we get
\begin{align}
\label{s33.2}
\int_1^{+\infty}y^{\wa}p(y)dy
=\int_1^{2\pi}y^{\wa}p(y)dy+\widetilde{II}^{\alpha}, \text{ with }
\widetilde{II}^{\alpha}\equiv\sum_{k=1}^{+\infty}\int_{2\pi k}^{2\pi(k+1)}y^{\wa}p(y)dy.
\end{align}
The series is absolutely convergent, due to the fact that $p(\cdot)$ is uniformly bounded and $\Re\alpha<-1$. The first integral at the right hand side has been computed in \eqref{Ialpha-2}.

Set
\begin{align}\label{e.6.2c}
\widetilde{II}^{\alpha}\equiv\sum_{k=1}^{+\infty}\int_{2\pi k}^{2\pi(k+1)}y^{\wa}p(y)dy.
\end{align}
By the change of variables $y\to y'=y-2\pi k$ and using the periodicity of $p$ we get
\begin{align}\label{g33.6.3}
\widetilde{II}^{\alpha}=\sum_{k=1}^{+\infty}\widetilde{II}^{\alpha}_k,
\end{align}
with
\[
\widetilde{II}_k^{\alpha}\equiv \int_0^{2\pi}(y+2\pi k)^{\wa}p(y)dy
=(2\pi k)^{\wa}\int_0^{2\pi}\Big(1+\frac{y}{2\pi k}\Big)^{\wa}p(y)dy,\ \
k\in\N.
\]
Using the binomial series expansion and 
 Lebesgue's dominated convergence, we get

\begin{align}\label{e.2.15a}\begin{split}
\widetilde{II}_k^{\alpha}=(2\pi k)^{\wa}
\sum_{l=0}^{+\infty}{\wa\choose l}\frac1{(2\pi k)^l}\int_0^{2\pi}y^lp(y)dy.
\end{split}\end{align}
Performing the integrals, using the definition \eqref{e.2.1} of $p(y)$ in $[0,2\pi)$, we 
get easily

\begin{align}\begin{split}\label{e.6.7}
\widetilde{II}^{\alpha}=\sum_{k=1}^{\infty}(2\pi k)^{\wa}
\left\{
\frac{\pi^2}{6}\cdot2\pi\sum_{l=0}^{\infty}{\wa\choose l}\frac1{l+1}\frac1{k^l}
-\frac{\pi}{2}(2\pi)^2\sum_{l=0}^{\infty}{\wa\choose l}\frac1{l+2}\frac1{k^l}\right.\\
\left.+\frac14(2\pi)^3\sum_{l=0}^{\infty}{\wa\choose l}\frac1{l+3}\frac1{k^l}
\right\}=\widetilde{A}(\alpha)
\end{split}\end{align}
(where the sums converge absolutely and we used the definition of $\widetilde{A}(\alpha)$
in Lemma \ref{l.2.1}).
Lemma \ref{l.2.1} follows then, introducing \eqref{Ialpha-2} resp. \eqref{e.6.7} into \eqref{s33.2}.
\end{proof}

\begin{lem}\label{prop6.2}
Let $\Real\alpha<-1$, then with the notation in Lemma \ref{l.2.1}:
\[
 \widetilde{A}(\alpha)=2^{\wa}\frac{\pi^{3+\wa}}{3}\sum_{l=0}^{+\infty}\frac1{l+1}
{\wa \choose l}\zeta(l-\wa)- 2^{1+\wa}\pi^{3+\wa}\sum_{l=0}^{+\infty}\frac1{l+2}
{\wa \choose l}\zeta(l-\wa)+ 2^{1+\wa}\pi^{3+\wa}\sum_{l=0}^{+\infty}\frac1{l+3}
{\wa \choose l}\zeta(l-\wa).
\]
\end{lem}
\begin{proof}
This follows from the expression for $\widetilde{A}(\alpha)$ 
in Lemma \ref{l.2.1} and the power series definition of $\zeta(z)$ with $ z=l-\alpha$, observing
that:
$\Real (l-\wa)\geqslant l-\Real\alpha>l+1\geqslant1$, $l\in\N_0$, so that the series over $k$ converge 
 absolutely to $\zeta(l-\alpha)$.
\end{proof}

\begin{lem}\label{lem6.3}
For $j=1,2,3$, $\Real\alpha<-2$:
\[
\sum_{l=0}^{\infty}\frac1{l+j}{\wa\choose l}\zeta(l-\wa)
=\lim_{\eta\downarrow0}\int_{\eta}^{1-\eta}\zeta(-\wa,1+t)t^{j-1}dt
=-\frac1{\wa+j}+\widetilde{A}_j(\wa),
\]
where
$\zeta(\rho,a)$ is the generalized Riemann zeta function, defined for
$a\notin-\N_0,\ \Real\rho>1$ by $\zeta(\rho,a)\equiv\sum_{n=0}^{\infty}\frac1{(n+a)^{\rho}}$ and
meromorphically continued to all $\rho\in\C$ (see, e.g., \cite[vol. 1, 24]{erd}, \cite{lau}),
and $\widetilde{A}_1(\wa)\equiv0, \widetilde{A}_2(\wa)\equiv\frac1{\wa+1}\zeta(-\wa-1), 
\widetilde{A}_3(\wa)\equiv\frac1{\wa+1}\zeta(-\wa-1)-\frac2{(\wa+1)(\wa+2)}\zeta(-\wa-2)$.
\end{lem}
\begin{proof}
From \cite[54.12.2, p. 359]{han}  or \cite[pp. 281, 286]{srich} :
\begin{align}\label{sn1}
\sum_{j=0}^{\infty}\frac{(\rho)_j}{j!}\zeta(\rho+j,a)t^j=\zeta(\rho,a-t),
\end{align}
for $|t|<|a|,\ \rho,\wa\in \C,\ a-t\notin-\N_0$.
Here $(\rho)_j\equiv\rho(\rho+1)\ldots(\rho+j-1)={\Gamma(\rho+j)}/{\Gamma(\rho)}$
is Pochhammer's symbol, $j\in\N,\ (\rho)_0\equiv1$.
From, e.g., \cite[p. 4]{Ma06} we have, for any $l\in\N$:
\[
{b\choose l}=(-1)^l\frac{\Gamma(l-b)}{l!\Gamma(-b)},\ b\in\C\setminus\N_0.
\]
From this and the definition of $(\rho)_j$ with $\rho=-b,\ j=l$, we have,
using $(-b)_l=\frac{\Gamma(-b+l)}{\Gamma(-b)}$:
\begin{align}\label{e.7.7.b}
{b\choose l}=(-1)^l\frac{\Gamma(l-b)}{l!\Gamma(-b)}
=\frac{(-1)^l}{l!}(-b)_l.
\end{align}
From this we get
\[
{b\choose l}\zeta(l-b)=\frac{(-1)^l}{l!}(-b)_l\zeta(l-b),
\]
hence
\begin{align}\label{n.2.1}
(-b)_l\zeta(l-b)=\frac{l!}{(-1)^l}{b\choose l}\zeta(l-b).
\end{align}
Using \eqref{sn1} for
$\rho=-\wa,\ j=l,\ a=1$ and $0<|t|<1$, we have on the other hand:
\begin{align}
\label{n.2.2}
\sum_{l=0}^{+\infty}\frac{(-\wa)_l}{l!}\zeta(l-\wa)t^l=\zeta(-\wa,1-t)
\end{align}
(where we used that $\zeta(l-\wa,1)=\sum_{n=0}^{+\infty}\frac1{(n+1)^{l-\wa}}
=\sum_{n'=1}^{+\infty}\frac1{(n')^{l-\wa}}=\zeta(l-\wa)$, where we set $n'=n+1$, 
 minding that $\Real(l-\alpha)>1$
since $l-\Real\alpha>1$).
From \eqref{n.2.1} (with $b=\wa$) we have
$(-\wa)_l\zeta(l-\wa)=\frac{l!}{(-1)^l}{\wa\choose l}\zeta(l-\wa)$
and we get
\[
\sum_{l=0}^{+\infty}\frac{(-\wa)_l}{l!}\zeta(l-\wa)t^l
=\sum_{l=0}^{+\infty}\frac1{l!}\frac{l!}{(-1)^l}{\wa\choose l}\zeta(l-\wa)t^l.
\]
Minding that the left hand side is the same as the left hand side of \eqref{n.2.2}
we get
\[
\sum_{l=0}^{+\infty}{\wa\choose l}\zeta(l-\wa)(-t)^l
=\zeta(-\wa,1-t).
\]
Replacing $t$ by $-t$ (which also satisfies $0<|t|<1$), and multiplying by $t^{j-1},\ j=1,2,3$, we get
\begin{align}\label{e.7.9.a}
\sum_{l=0}^{+\infty}{\wa\choose l}\zeta(l-\wa)t^{l+j-1}
=\zeta(-\wa,1+t)t^{j-1}.
\end{align}
Integrating with respect to $t$ on $[\eta,1-\eta],\ \frac12>\eta>0$ we get
\begin{align}\label{st.40}
\int_{\eta}^{1-\eta}\sum_{l=0}^{+\infty}{\wa\choose l}\zeta(l-\wa)t^{l+j-1}dt
=\int_{\eta}^{1-\eta}\zeta(-\wa,1+t)t^{j-1}dt.
\end{align}

We shall now write $\zeta(l-\wa)=[\zeta(l-\wa)-1+1]$
and insert this into the left hand side of \eqref{st.40}
\begin{align}\label{st.41}
\int_{\eta}^{1-\eta}\sum_{l=0}^{+\infty}{\wa \choose l}[\zeta(l-\wa)-1]t^{l+j-1}dt
+\int_{\eta}^{1-\eta}\sum_{l=0}^{+\infty}{\wa \choose l}t^{l+j-1}dt.
\end{align}

Using that from \eqref{n.2.1} with $b=\wa$ we have 
$|(-\wa)_l\zeta(l-\wa)|=l!\left|{\wa\choose l}\right|\left|\zeta(l-\wa)\right|$,
we get easily for the $N$-th approximation of the integrand of the first term in \eqref{st.41}: 
\begin{align}\begin{split}\label{n.3.b}
&\left|\sum_{l=0}^N{\wa\choose l}(\zeta(l-\wa)-1)\right|
\leqslant\sum_{l=0}^{N}\frac1{l!}|(-\wa)_l||\zeta(l-\wa)-1|\\
&\leqslant\sum_{l=0}^{\infty}\frac1{l!}\sqrt{(\Real\alpha)^2+(\im\alpha)^2}
\ldots\sqrt{(\Real (\alpha+(l-1)))^2+(\im\alpha)^2}
\zeta(l-\Real\alpha),
\end{split}\end{align}
where we used the definition of $(-\wa)_l$ and
$|(-\wa+q)|=\sqrt{(\Real\alpha+q)^2+(\im\alpha)^2}, q=0,\ldots,l-1$,
together with:
\[
|\zeta(l-\wa)-1|\leqslant
\sum_{n=2}^{\infty}\Big|\frac1{n^{\Real\alpha+l}}\Big|
=\sum_{n=2}^{\infty}\frac1{n^{\Real\alpha}}
\leqslant[\zeta(l-\Real\alpha)-1],\ \ l\in\N_0
\]
(because $|n^{l-\alpha}|=n^{l-\Real\alpha}$, since
$|n^{-iv}|=|e^{-iv\ln n}|=1$, $n\in\N$, $v\in\R$ and
the series being absolutely convergent, since $l-\Real\alpha>2$, for all $l\in\N_0$).
But
\[
\zeta(l-\Real\alpha)-1\leqslant\frac{l+1-\Real\alpha}{l-1-\Real\alpha}\frac1{2^{l-\Real\alpha}}
\]
because $\zeta(\sigma)\leqslant1+\frac{\sigma+1}{\sigma-1}\frac1{2^\sigma}$, for $\sigma>1$
(see, e.g., \cite[Ex. 2, p. 48]{jam}).
Hence we get from \eqref{n.3.b} the following bound:
\[
\bigg|\sum_{l=0}^{N}{\wa \choose l}(\zeta(l-\wa)-1)\bigg|
\leqslant\sum_{l=0}^{\infty}\frac1{l!}\big(\sqrt{[(\Real\alpha)+(l-1)]^2+(\im\alpha)^2}\big)^l
\frac{l+1-\Real\alpha}{l-1-\Real\alpha}\frac1{2^{l-\Real\alpha}}.
\]
The series converges absolutely, as seen from Stirling's formulae.
The bound \eqref{n.3.b} is then finite, independent of $N$. Using Lebesgue's dominated convergence
theorem it is 
then not difficult to
 prove that
\begin{align}\label{bo.41}
\sum_{l=0}^{+\infty}{\wa\choose l}\zeta(l-\wa)\frac1{l+j}[(1-\eta)^{l+j}-\eta^{l+j}]
=\int_{\eta}^{1-\eta}\zeta(-\wa,1+t)t^{j-1}dt,\ j=1,2,3.
\end{align}
The summand on the left hand side is bounded absolutely uniformly
in $\eta$ by $2$
(since $(1-\eta)^{l+i}\leqslant1$ and 
$\eta^{l+j}\leqslant1/{2^{l+j}}\leqslant1,\ l\in\N_0,\ j=1,2,3$).
By a discrete version of the Lebesgue dominated convergence theorem we can interchange 
the limit $\eta\downarrow0$ with the summation,
getting that the limit for $\eta\downarrow0$ of the left hand side of \eqref{bo.41}
is equal to 
$\sum_{l=0}^{+\infty}{\wa \choose l}\zeta(l-\wa)\frac1{l+j}$,
 which is the left hand side in the formula in Lemma \ref{lem6.3}.

On the right hand side of \eqref{bo.41} we have, using the definition
of $\zeta(-\wa,1+t)$ in Lemma \ref{lem6.3}:
\begin{align}
\label{n.4.m}
\int_{\eta}^{1-\eta}\zeta(-\wa,1+t)t^{j-1}dt
=\int_{\eta}^{1-\eta}\sum_{k=0}^{\infty}(k+(1+t))^{\wa}t^{j-1}dt.
\end{align}
It is not difficult to convince ourselves that one can interchange the sum and the integral obtaining
\begin{align}\label{e.7.16}
\lim_{\eta\downarrow0}\sum_{k=0}^{+\infty}\int_{\eta}^{1-\eta}(k+(1+t))^{\wa}dt
=-\frac{1}{\wa+1}.
\end{align}
Similarly, for $j=2$ we have as a result:
\beq
\label{e.7.22b}
\begin{aligned}
\sum_{k=0}^{+\infty}\int_{\eta}^{1-\eta}(k+(1+t))^{\wa}t\ dt
=&\sum_{k=0}^{+\infty}\int_{k+1+\eta}^{k+2-\eta}t'^{\wa}(t'-k-1)dt'\\
=&\sum_{k=0}^{+\infty}\int_{k+1+\eta}^{k+2-\eta}t'^{\wa+1}dt'
-\sum_{k=0}^{+\infty}(k+1)\int_{k+1+\eta}^{k+2-\eta}t'^{\wa}dt'.
\end{aligned}
\eeq
Using the above result for $j=1$ with $\wa$ replaced by $\wa+1$ we see that the first sum
on the right hand side converges for $\eta\downarrow0$ to $-1/(\wa+2)$.

As for the $N$-th approximation of the second sum on the right hand side of \eqref{e.7.22b} we have
\begin{align*}
-\sum_{k=0}^N(k+1)\int_{k+1+\eta}^{k+2-\eta}t'^{\wa}dt'
&=-\sum_{k=0}^N\frac{(k+1)}{\wa+1}[(k+2-\eta)^{\wa+1}-(k+1+\eta)^{\wa+1}]\\
&\stackrel{N\to+\infty}{\longrightarrow}
-\sum_{k=0}^{+\infty}\frac{(k+1)}{\wa+1}[(k+2-\eta)^{\wa+1}-(k+1+\eta)^{\wa+1}]\\
&\stackrel{\eta\downarrow0}{\longrightarrow}
-\sum_{k=0}^{+\infty}\frac{(k+1)}{\wa+1}[(k+2)^{\wa+1}-(k+1)^{\wa+1}]\\
&=\frac1{\wa+1}\sum_{k=1}^{\infty}k^{\wa+1}=\frac1{\wa+1}\zeta(-\wa-1),
\end{align*}
where we used $(\Real\alpha+1)<-1$. Hence
\begin{align}\label{e.7.17}
\sum_{k=0}^{N}\int_{\eta}^{1-\eta}(k+(1+t))^{\wa}t\ dt
\stackrel{N\to+\infty}{\longrightarrow}
\sum_{k=0}^{+\infty}\int_{\eta}^{1-\eta}(k+(1+t))^{\wa}t\ dt
\stackrel{\eta\downarrow0}{\longrightarrow}-\frac1{\wa+2}
+\frac1{\wa+1}\zeta(-\wa-1).
\end{align}
For $j=3$ one first considers 
\[
\rho_{N,\eta}(\betaw)\equiv\sum_{k=0}^{N}\int_{\eta}^{1-\eta}(k+(1+t))^{\betaw}t^2\ dt,
\]
for $\Real\betaw<-3$ 
(and then one replaces ``by analytic continuation'' $\betaw$ by $\wa$).
By careful considerations similar to these we did for $j=1,2$, we obtain
\[
\lim_{N\to\infty}\lim_{\eta\downarrow0}\rho_{N,\eta}(\beta)=
-\frac1{\beta+3}+\frac1{\beta+1}\zeta(-\beta-1)
-\frac2{(\beta+1)(\beta+2)}\zeta(-\beta-2).
\]
By analytic continuation one then obtains

\beq
\label{e.7.22}
\lim_{\eta\downarrow0}\sum_{k=0}^{\infty}\int_{\eta}^{1-\eta}(k+(1+t))^{\wa}t^2\ dt
=-\frac1{\wa+3}+\frac1{\wa+1}\zeta(-\wa-1)
-\frac2{(\wa+1)(\wa+2)}\zeta(-\wa-2).
\eeq
The completion of the proof is obtained from
\eqref{e.7.16}, \eqref{e.7.17}, \eqref{e.7.22}
combined with the fact that the limit of \eqref{e.7.9.a}
for $\eta\downarrow0$ is the left hand side in the expression in Lemma \ref{lem6.3} and observing that \eqref{n.2.2}
holds.
\end{proof}

\begin{cor}\label{cor.6.5}
$\widetilde{A}$ as defined in Lemma \ref{l.2.1} satisfies, for $\Real\alpha<-2$:
\begin{align*}
\widetilde{A}(\alpha)=2^{\wa}\pi^{3+\wa}\Big(-\frac13\frac1{\wa+1}+\frac2{\wa+2}-\frac2{\wa+3}\Big)
-\frac{(2\pi)^{3+\wa}}{2(\wa+1)(\wa+2)}\zeta(-\wa-2).
\end{align*}
\end{cor}
\begin{proof}
This is immediate from Lemma \ref{prop6.2} and Lemma \ref{lem6.3}.
\end{proof}

\begin{cor}\label{cor.6.5a}
Let $D(\alpha)$ be as in Lemma \ref{l.2.1}. Then for $\Real\alpha<-2$:
\[
\begin{aligned}
D(\alpha)=&
-\frac{\pi^2}{6}\frac1{\wa+1}+\frac\pi2\frac1{\wa+2}
-\frac14\frac1{\wa+3}
+
\frac{2^{\wa+1}}{2\cdot3}\frac1{\wa+1}\pi^{3+\wa}
-\frac{\pi^{3+\wa}}{2}\frac{2^{\wa+2}}{\wa+2}+\frac1{2^2}\frac1{\wa+3}
2^{\wa+3}\pi^{3+\wa}
\widetilde{A}(\alpha)\\
=&
-\frac{\pi^2}{6}\frac1{\wa+1}+\frac\pi2\frac1{\wa+2}-\frac14\frac1{\wa+3}
-\frac{(2\pi)^{3+\wa}}{2(\wa+1)(\wa+2)}\zeta(-\wa-2).
\end{aligned}
\]
\end{cor}
This coincides with formula \eqref{eq-cor-Dal}, and hence yields a new proof  of Corollary \ref{cor-Dal}.
\begin{proof}
This is immediate from Lemma \ref{l.2.1} and Corollary \ref{cor.6.5},
due to a compensation of terms.
\end{proof}

%%%
%%%%%BIBLIOGRAFIA SU LINUX UNI-BONN 
%\bibliographystyle{/home/cacciapuoti/Dropbox/WIP-dropbox/myamsplain}
% \bibliography{/home/cacciapuoti/Dropbox/WIP-dropbox/mywwb11}

%%%%%BIBLIOGRAFIA SUL MIO POWERBOOK
%\bibliographystyle{/Users/claudio/Dropbox/works-dropbox/myamsalpha}
%\bibliography{dilogarithm}

%%%%%BIBLIOGRAFIA DA FILE SPECIFICO
%\bibliography{dilogarithm}

\end{document}